\documentclass{article}

\usepackage{arxiv}
\makeatletter
\AtBeginDocument{%
  \pagestyle{fancy}%
  \rhead{}%
}

\makeatother
\usepackage{booktabs}
\usepackage{tabularx}
\usepackage[utf8]{inputenc} 
\usepackage[T1]{fontenc}    
\usepackage{amssymb} 
\usepackage{amsthm} 

\usepackage{url}            
\usepackage{amsfonts}       
\usepackage{microtype}      
\usepackage{graphicx}
\usepackage{amsmath}
\usepackage{array} 
\usepackage{natbib}
\usepackage{lineno}         
\usepackage{longtable}
\usepackage{pdflscape}

\usepackage{tikz}
\usepackage{tikz-3dplot}
\usetikzlibrary{arrows.meta,positioning,calc,shapes.symbols}
\definecolor{diagramblue}{RGB}{13,101,139}

\usepackage{authblk}

\makeatletter
\renewcommand\AB@affilnote[1]{\textsuperscript{#1}\,}
\makeatother

\usepackage[colorlinks=true,linkcolor=black,citecolor=blue,urlcolor=blue]{hyperref}

\usepackage{enumitem}
\setlist[itemize]{leftmargin=1.6em}
\setlist[enumerate]{leftmargin=1.8em}

\newtheorem{definition}{Definition}

\newtheorem{proposition}{Proposition}
\newtheorem{theorem}{Theorem}

\newcommand{\Sys}{\mathsf{Sys}}
\newcommand{\Traj}{\mathsf{Traj}}
\newcommand{\Time}{\mathsf{T}}
\newcommand{\Ob}{\mathrm{Ob}}
\newcommand{\Hom}{\mathrm{Hom}}
\newcommand{\id}{\mathrm{id}}

\newcommand{\diam}{\diamondsuit}

\title{A Category Theory Account of AI Identity}

\author[1,2,3]{Andrea Ferrario} 
\affil[1]{Institute of Biomedical Ethics and History of Medicine, University of Z\"urich, Z\"urich, Switzerland}
\affil[2]{SUPSI, Dalle Molle Institute for Artificial Intelligence (IDSIA), Lugano, Switzerland}
\affil[3]{ETH Z\"urich, Z\"urich, Switzerland}

\begin{document}

\maketitle

\begin{abstract}
Artificial intelligence (AI) systems are routinely modified after deployment through retraining, reconfiguration, and changes in their operational environments. These transformations raise a metaphysical question with direct consequences for AI governance: under what conditions does an AI system remain the same system over time or across deployments? Earlier trustworthiness-based work formulates synchronic and diachronic identity propositionally, by relating identity within a fixed AI system type to equality of trustworthiness levels. Such criteria specify when identity statements are true, but leave implicit the structure of the states compared, the transformations connecting them, and the temporal organization of persistence. We develop a category-theoretic formalization of AI identity. An AI system type is specified by a datum \(\diam=(F,P,L_P)\), consisting of a techno-function, a trustworthiness profile, and a trustworthiness-level function. Profile-relative states are connected by admissible lifecycle paths, which are restricted to trustworthiness-level-preserving transformations and quotiented to obtain a thin reachability category. Temporally admissible functors represent AI system histories, while time-synchronous natural transformations compare realized histories. The formalization yields two categorical interpretations of the earlier propositional criteria. A weak interpretation recovers identity, within a fixed datum \(\diam\), as equality of trustworthiness level. A strong interpretation refines it by requiring mutual trustworthiness-preserving reachability, expressed through state isomorphism or natural isomorphism of realized histories. Category theory therefore replaces a single undifferentiated identity relation with a structured hierarchy of diachronic and synchronic criteria. The resulting framework identifies identity-related preconditions for transferring responsible AI claims, evidence, and governance procedures across versions or deployments, without treating categorical identity as sufficient by itself for such transfer.
\end{abstract}


\section{Introduction}
\label{section:introduction}
The idea that entities require identity criteria for ontological respectability has a long philosophical tradition, canonically expressed by Quine's dictum ``no entity without identity'' \citep[p.~44]{quine1969ontological}. Identity criteria specify the conditions under which an entity at one time or in one setting is the same as, or different from, an entity at another. Without such criteria, it remains unclear what is being counted, compared, governed, or held responsible across change.

Artificial intelligence (AI) systems make this longstanding problem especially pressing. Deployed AI systems are routinely modified through retraining, fine-tuning, recalibration, threshold adjustment, model replacement, rollback, data refresh, pipeline reconfiguration, monitoring interventions, and changes in their operational environments \citep{kreuzberger2023machine}. These transformations may be necessary to correct errors, respond to distribution shift, improve fairness or robustness, or adapt a system to changing operational or responsible AI requirements. Some systems are accessed globally through large numbers of personalized or deployment-specific instances. These changes and apparent multiplicity raise the metaphysical question: \emph{when does a changing AI system remain the same system?} And, further, \emph{when are AI systems the same system?} These classical metaphysical questions have immediate consequences for AI governance. A performance evaluation, fairness analysis, conformity assessment, or post-market monitoring result is produced for a particular system under particular conditions \citep{EU_AI_Act_2024}. Applying it to a later version or a different deployment presupposes that the relevant object has remained \emph{sufficiently} identical. Similarly, regulatory concepts such as substantial modification distinguish changes that preserve the governed system from those that may generate a new object of evaluation or responsibility \citep{EU_AI_Act_2024,ferrario_update2026,ferrario2026high}.

Ferrario recently addressed this philosophical problem by adapting the \texttt{function}\textsuperscript{\tiny +} account of \citet{carrara2009fine} to the metaphysics of AI systems \citep{ferrario2025trustworthiness}. On this account, artifact kinds are fixed by their techno-functions, while the identity and persistence of their instances additionally depend on the operational principles and admissible configurations through which those functions are correctly realized. For AI systems, the relevant operational principle is expressed through \emph{trustworthiness}: the collection of performance, robustness, fairness, explainability, safety, security, auditability, oversight, and related commitments that a system must satisfy in order to function \emph{correctly}. Trustworthiness thereby provides a governance-relevant interpretation of the contextual and normative embedding that artifact metaphysics such as the \texttt{function}\textsuperscript{\tiny +} account by \citet{carrara2009fine} treats as constitutive of technological identity \citep{HLEG2019,EU_AI_Act_2024}. On this basis, Ferrario introduced synchronic and diachronic criteria of AI identity that depend on how the trustworthiness of AI systems is operationalized and measured over time \citep{ferrario2025trustworthiness}. However, the criteria proposed in \citet{ferrario2025trustworthiness} are propositional: they specify synchronic and diachronic identity through biconditionals involving a fixed AI system type and equality of trustworthiness levels. They thereby provide truth conditions for identity claims at selected times, but do not yet represent the internal structure of the compared states, the admissible transformations connecting them, the directionality and composition of those transformations, or the temporal coherence of complete histories. In particular, equality of trustworthiness levels does not distinguish between states connected in one direction, states connected in both directions, and states between which no admissible transformation exists.

In this work, we formalize Ferrario's trustworthiness-based account of AI identity using category theory. Category theory is well suited to this task because it characterizes objects through their admissible transformations, invariants, compositions, and structural relations rather than through component-wise equality alone \citep{eilenberg1945general,maclane1963natural,mac1971categories}. Historically, it emerged from the study of natural equivalences and may be understood as extending the structural perspective of Klein's \emph{Erlanger Programm}: what matters is not merely what an object contains, but which transformations preserve the structure regarded as essential \citep{marquis2008geometrical}. This perspective is particularly appropriate for AI systems, whose persistence may consist precisely in preserving governance-relevant structure through substantial material, computational, and organizational change. Thus, category theory makes it possible to lift the propositional criteria into a relational and temporal structure. The resulting categorical framework supports two interpretations of AI identity. The \emph{weak interpretation} recovers  Ferrario's original propositional criteria: within a fixed type datum, states are identified whenever they belong to the same trustworthiness-level fibre. 
The \emph{strong interpretation} adds a transformation-grounded condition: AI states are identical only when they are mutually reachable through admissible trustworthiness-level-preserving paths. At the level of realized AI system history functors, the corresponding strong criterion is natural isomorphism witnessed by time-synchronous comparison morphisms. Thus, the strong interpretation of AI identity is a stricter categorical refinement made available by representing the transformations and histories that the propositional formulation leaves implicit.

Following \citet{ferrario2025trustworthiness}, our category-theoretic construction begins with an AI system type datum
\(\diam=(F,P,L_P)\), consisting of a techno-function \(F\), a trustworthiness profile \(P\), and a trustworthiness-level function \(L_P\). Profile-relative states combine quantified assessments \(q_P\in I_P\) of the trustworthiness profile with their associated levels \(L_P(q_P)\). Admissible AI lifecycle transformations generate a path category, which is restricted to trustworthiness-level-preserving paths and quotiented by an equivalence relation that abstracts from their concrete provenance. The resulting thin category \(\Sys_\diam\) records directed trustworthiness-preserving reachability between states. 
Post-deployment AI system histories are represented by temporally admissible functors
\[
\mathsf A_\diam:\Time_{t_0}\to\Sys_\diam.
\]
For a fixed observation time \(T\geq t_0\), their restrictions to the interval \([t_0,T]\), when instantiated by at least one deployed AI system realization---equivalently: \emph{token}---define \emph{realized AI system histories}. These form the objects of the category
\[
\Traj_{\diam,[t_0,T]},
\]
whose morphisms are time-synchronous natural transformations comparing realized histories through states occupied at the same time. Ferrario's AI identity criteria are then retrieved and expanded by studying identity and isomorphism in \(\Sys_\diam\), and natural isomorphism between realized AI system histories in \(\Traj_{\diam,[t_0,T]}\). Our formalization is neutral among competing metaphysical theories of persistence \citep{lowe1983identity,baker2004ontology,williamson2013identity}. Its objects are profile-relative states and realized AI system histories rather than complete temporal parts or individuated concrete AI system tokens. The functorial account may therefore be interpreted either as organizing time-relative stages or as representing the changing properties of enduring systems. Its central commitment is that AI system identity under change depends on temporally ordered states, admissible lifecycle transformations, realized histories, and explicitly specified trustworthiness invariants. 

This work makes three principal contributions.
\begin{enumerate}[label=(\roman*),nosep]
\item It provides a categorical reconstruction of trustworthiness-based AI system identity by formalizing profile-relative states, admissible lifecycle paths, trustworthiness-preserving reachability, and realized
AI system histories.

\item It distinguishes a weak level-theoretic identity criterion from a strong transformation-grounded criterion. The weak criterion recovers equality of trustworthiness level \citep{ferrario2025trustworthiness}, whereas the strong criterion requires mutual trustworthiness-preserving reachability, time-synchronously witnessed for comparisons of realized histories.

\item It identifies identity-related preconditions for considering the transfer of responsible AI evidence, claims, and governance procedures across versions and deployments, without treating categorical identity as sufficient for unrestricted transfer.
\end{enumerate}


The remainder of the paper proceeds as follows. Section~\ref{section:categorical_prelim} introduces the required categorical notions. Section~\ref{section:AI_meta_TW} presents the trustworthiness-based metaphysical foundations. Section~\ref{section:cat_AI_identity} constructs the state category \(\Sys_\diam\), and Section~\ref{section:AI_identity_functors} introduces temporally admissible AI system histories and the category of realized AI system histories \(\Traj_{\diam,[t_0,T]}\). Section~\ref{section:AI_identity_category} develops the weak and strong identity criteria and illustrates them through examples. The final sections discuss the philosophical and governance implications of the categorical framework.

\section{Categorical Preliminaries}
\label{section:categorical_prelim}
This section introduces the minimal categorical notions used in the remainder of the paper: categories, isomorphisms, functors, natural transformations, thin categories, categorical congruences, quotient categories, and maximal subgroupoids. We refer to standard sources for examples and further details \citep{eilenberg1945general,mac1971categories,spivak2014category,yanofsky2024monoidal}.

\begin{definition}[Category]
A \emph{category} $\mathcal{C}$ consists of:
\begin{itemize}[itemsep=0pt]
    \item a collection of objects $\Ob(\mathcal{C})$;
    \item for each $X,Y\in\Ob(\mathcal{C})$, a collection of morphisms $\Hom_{\mathcal{C}}(X,Y)$;
    \item for each object $X$, an identity morphism $\id_X\in\Hom_{\mathcal{C}}(X,X)$;
    \item a composition operation
    \[
    \circ:\Hom_{\mathcal{C}}(Y,Z)\times \Hom_{\mathcal{C}}(X,Y)\to \Hom_{\mathcal{C}}(X,Z),
    \]
    satisfying associativity and unit laws.\footnote{Following \citep{yanofsky2024monoidal}, we use ``collections'' for $\Ob(\mathcal{C})$ and $\Hom_{\mathcal{C}}(X,Y)$ instead of ``sets'' or ``classes''. Nothing in the present paper depends on a particular choice of set-theoretic foundation or on size distinctions between small and large categories.}
\end{itemize}
\end{definition}

\begin{definition}[Isomorphism]
A morphism \(f:A\to B\) in a category \(\mathcal C\) is an \emph{isomorphism} if there exists a morphism \(g:B\to A\) such that
\[
g\circ f=\id_A,
\qquad
f\circ g=\id_B.
\]
The morphism \(g\) is then uniquely determined and is called the inverse of \(f\).
\end{definition}

Two objects \(A\) and \(B\) of \(\mathcal C\) are \emph{isomorphic} if there exists an isomorphism \(f:A\to B\) in \(\mathcal C\). In that case, we write
\[
A\cong_{\mathcal C}B.
\]

\begin{definition}[Functor]
A \emph{functor} $\mathsf F:\mathcal{C}\to\mathcal{D}$ assigns objects of $\mathcal{C}$ to objects of $\mathcal{D}$ and morphisms of $\mathcal{C}$ to morphisms of $\mathcal{D}$, preserving identities and composition:
\[
\mathsf F(\id_X)=\id_{\mathsf F (X)},\qquad \mathsf F(g\circ f)=\mathsf F(g)\circ \mathsf F(f).
\]
\end{definition}

Maps between functors are called natural transformations:

\begin{definition}[Natural transformation]
Given functors $\mathsf F,\mathsf G:\mathcal{C}\to\mathcal{D}$, a \emph{natural transformation}
\(\eta:\mathsf F\Rightarrow \mathsf G\)
is a family of morphisms
\(
\{\eta_X:\mathsf F(X)\to \mathsf G(X)\}_{X\in\Ob(\mathcal{C})}
\)
such that, for every morphism $f:X\to Y$ in $\mathcal{C}$, the following naturality condition holds:
\[
\mathsf G(f)\circ \eta_X=\eta_Y\circ \mathsf F(f).
\]
\end{definition}

\begin{definition}[Natural isomorphism]
Let \(\mathsf F,\mathsf G:\mathcal C\to\mathcal D\) be functors. A natural transformation
\(\eta:\mathsf F\Rightarrow\mathsf G\)
is a \emph{natural isomorphism} if every component
\[
\eta_X:\mathsf F(X)\to\mathsf G(X)
\]
is an isomorphism in \(\mathcal D\). Equivalently, there exists a natural transformation
\(\eta^{-1}:\mathsf G\Rightarrow\mathsf F\)
such that
\[
\eta^{-1}\circ\eta=\id_{\mathsf F},
\qquad
\eta\circ\eta^{-1}=\id_{\mathsf G}.
\]
In this case, the functors \(\mathsf F\) and \(\mathsf G\) are said to be \emph{naturally isomorphic}, written
\(\mathsf F\cong\mathsf G\).
\end{definition}

\begin{definition}[Functor category]
Let \(\mathcal C\) and \(\mathcal D\) be categories. The
\emph{functor category} \([\mathcal C,\mathcal D]\)
has functors \(\mathcal C\to\mathcal D\) as objects and natural transformations as morphisms. Identities and composition are defined componentwise.
\end{definition}

The following categorical notions are important.

\begin{definition}[Quotient category]
Let \(\mathcal C\) be a category, and suppose that each hom-set
\(\Hom_{\mathcal C}(X,Y)\)
is equipped with an equivalence relation \(\sim\) compatible with composition: whenever
\(f\sim f'\) and \(g\sim g'\), with the relevant composites defined, then
\(g\circ f\sim g'\circ f'.\) The \emph{quotient category} \(\mathcal C/{\sim}\) has the same objects as \(\mathcal C\), and its morphisms are equivalence classes
\[
\Hom_{\mathcal C/{\sim}}(X,Y)
=
\Hom_{\mathcal C}(X,Y)/{\sim}.
\]
Identities and composition are defined by
\[
\id_X=[\id_X],
\qquad
[g]\circ[f]=[g\circ f].
\]
Compatibility with composition ensures that these operations are well defined.
\end{definition}

\begin{definition}[Thin category] A category \(\mathcal C\) is \emph{thin} if every hom-set contains at most one morphism. 
\end{definition} 

\begin{proposition}[Commutativity in thin categories]
\label{prop:thin_categories_commute}
Let \(\mathcal C\) be a thin category. Then every diagram in \(\mathcal C\) commutes whenever all the morphisms and composites occurring in the diagram exist.
\end{proposition}

\begin{proof}
Any two paths in the diagram with the same source and target determine parallel composite morphisms. Since \(\mathcal C\) is thin, the hom-set between any two objects contains at most one morphism. Hence the two composites must be equal.
\end{proof}

Note that a thin category is equivalently a preorder represented categorically. Given a preorder \((X,\leq)\), define a category \(\mathcal C_{(X,\leq)}\) with objects the elements of \(X\) and
\[
\Hom_{\mathcal C_{(X,\leq)}}(x,y)=
\begin{cases}
\{\star_{x,y}\} & \text{if } x\leq y,\\
\varnothing & \text{otherwise.}
\end{cases}
\]
Reflexivity gives the identity morphisms, transitivity gives composition, and each hom-set contains at most one morphism.

\begin{definition}[Core of a category] The \emph{core}, or maximal subgroupoid, of a category \(\mathcal C\), denoted by \(\operatorname{Core}(\mathcal C)\), is the subcategory with the same objects as \(\mathcal C\) and only the isomorphisms of \(\mathcal C\) as morphisms. \end{definition}

\section{Trustworthiness-Based Metaphysics of AI Systems}
\label{section:AI_meta_TW}
Defining AI systems is notoriously difficult because the term encompasses a wide variety of software- and hardware-based artifacts. For the purposes of this paper, we adopt the definition provided in Article~3 of the EU AI Act. Accordingly, an AI system is

\begin{quote}
a machine-based system that is designed to operate with varying levels of autonomy and that may exhibit adaptiveness after deployment, and that, for explicit or implicit objectives, infers, from the input it receives, how to generate outputs such as predictions, content, recommendations, or decisions that can influence physical or virtual environments (Art.~3, \citep{EU_AI_Act_2024}).
\end{quote}

This deliberately broad definition encompasses, among others, medical systems that predict pathophysiological states such as sepsis, decision-support systems used in financial services for credit lending or know-your-customer procedures, and conversational agents based on large language models. Having fixed the class of systems under consideration, we now turn to the property that anchors their identity in the present account: trustworthiness.

\subsection{Trustworthiness as an Anchor for AI System Identity}
\label{subsection:TW_identity_anchor}

\subsubsection{The \texttt{function}\textsuperscript{\tiny +} Account of Technical Artifacts}
Our account of AI identity builds on the \texttt{function}\textsuperscript{\tiny +} metaphysics of technical artifacts developed by \citet{carrara2009fine}. The motivation for their account lies in a longstanding dispute about whether artifacts possess genuine identity and persistence conditions. On a traditional anti-realist view of artifact kinds, artifacts do not have the metaphysical standing of natural entities. Natural things, such as organisms, appear to come with internal principles of development, activity, maintenance, and decay that help determine when they begin to exist, persist, and cease to exist. Artifacts, by contrast, seem to depend on human intentions, practices, and classifications. They ``may
not seem to be supplied with well-defined or well-grounded \emph{persistence conditions}'' \citep[p.~17, emphasis in original]{lowe2014how}. 
A related difficulty concerns the basis on which artifacts should be individuated. If artifacts are individuated by their material parts, then ordinary repair, replacement, redesign, or reconfiguration threaten their persistence, as the ship of Theseus puzzle shows \citep{hobbes1655corpore}. If they are individuated only by function, then their identity criteria seem too coarse-grained \citep{wiggins2001sameness}: very different objects may perform the same function, and function alone therefore seems insufficient to determine what kind of artifact something is, or whether it remains the same artifact over time.

\citet{baker2004ontology} and \citet{elder2004real} resist this anti-realist tendency by defending the metaphysical reality of artifacts and by treating artifact functions as identity-relevant. Carrara and Vermaas accept this realist re-orientation, but refine it with their \texttt{function}\textsuperscript{\tiny +} account.  Their central point is that artifact identity cannot be fixed either by material constitution alone or by function alone. It requires a ``conjunction'' of function and constraints on its admissible realization. This is the role of the ``{\tiny +}'' in \texttt{function}\textsuperscript{\tiny +}: an artifact kind is determined by a \emph{techno-function}, namely the technical capacity that artifacts of that kind are designed to realize, together with an operational principle and a normal configuration specifying how that function is to be correctly realized in practice.\footnote{Techno-functions can be specified at different levels of detail: broader specifications pick out wider artifact kinds, whereas finer specifications carve out narrower ones.}
The \texttt{function}\textsuperscript{\tiny +} account therefore permits substantial variation in material composition and technical realization. Two artifacts may instantiate the same kind and persist through change even when their components differ, provided that they retain the relevant techno-function and continue to realize it through admissible operational principles and configurations. The \texttt{function}\textsuperscript{\tiny +} framework is thus more permissive than mereological essentialism, but more discriminating than function-only individuation. In doing so, it explains how artifacts can have genuine metaphysical standing while remaining design-dependent, context-sensitive, and open to material variation.

\subsubsection{Techno-Functions, Trustworthiness Profiles, and Level Functions}
\citet{ferrario2025trustworthiness} adapts Carrara and Vermaas' \texttt{function}\textsuperscript{\tiny +} account to AI systems.  The idea behind this maneuver is that identity of AI systems cannot be fixed by material or computational constitution alone: models, datasets, interfaces, deployment environments, documentation, monitoring procedures, and organizational arrangements may change while the system remains functionally continuous. Yet function alone is also too coarse-grained. Many AI systems may share a nominal function, such as classification, prediction, recommendation, or content generation, while differing substantially in the constraints, safeguards, performance expectations, and operational conditions under which they count as appropriately functioning. The \texttt{function}\textsuperscript{\tiny +} framework addresses this problem by preserving the centrality of designed AI function while requiring further criteria that specify how that function is realized and assessed in practice. 

Applying \texttt{function}\textsuperscript{\tiny +}
to AI requires specifying the \emph{techno-function} \(F\) of these systems, namely, the goal-directed technical capability that an AI system is designed to realize. It is related to, but not identical with, the system's \emph{intended use} in regulatory terminology---see Article~3 of the EU AI Act \citep{EU_AI_Act_2024}. The intended use describes the use for which the provider presents the system, including the relevant context and conditions of use. The techno-function captures the technical capability through which that use is made possible.
\footnote{The distinction is especially important for generative and multi-purpose systems such as large language models. A broadly specified techno-function, such as next-token prediction or autocompletion, may support many intended uses, including drafting, summarization, tutoring, translation, or entertainment. In the present account, the techno-function anchors identity at the level of AI system kind, while the trustworthiness profile and level function capture the contextualized conditions under which that function counts as appropriately realized for a given intended use.}
Thus, a system whose intended purpose is to support consumer-credit decisions may have the techno-function of inferring individualized credit-risk scores from applicant data according to a specified modeling pipeline. A system whose intended purpose is emergency-room decision support may have the techno-function of predicting a patient's risk of sepsis from clinical observations. \(F\) characterizes the AI system kind at the level of designed capability, specifying what the system is meant to achieve, through which kinds of inputs, inferential procedures, models, and outputs.

For AI systems, the operational principle discussed within \texttt{function}\textsuperscript{\tiny +} is expressed through \emph{trustworthiness}: the collection of performance, robustness, fairness, explainability, safety, security, auditability, oversight, and related requirements that a system must satisfy in order to function correctly \citep{EU_AI_Act_2024,HLEG2019}. These requirements constrain admissible implementations without fixing a unique model, software stack, hardware configuration, or organizational arrangement. These requirements are collected into a trustworthiness profile.

\begin{definition}[Trustworthiness profile, \citep{ferrario2025trustworthiness}]
\label{def:TW_profile}
Let \(F\) be a techno-function. A \emph{trustworthiness profile} for \(F\) is a finite specification
\[
P:=\{p_1,\dots,p_n\}
\]
of trustworthiness dimensions, together with the requirements, measurement conventions, evidential conditions, and aggregation procedures through which those dimensions are assessed.
\end{definition}

The profile \(P\) also incorporates the intended-purpose, deployment, and institutional conditions under which these dimensions are interpreted. It therefore fixes the level of abstraction at which identity and persistence are evaluated. Typical dimensions include predictive performance, robustness, fairness, explainability, safety, security, auditability, and human oversight. Their relevance for AI governance is contextual: it depends, in particular, on the intended use of the system and the risk it poses to individuals, organizations, and society \citep{EU_AI_Act_2024}. Further, their definition and operationalization are possible at different levels of abstraction, according to contextualized standards of practice and governance cultures. However, these dimensions must be made measurable as trustworthiness should be operationalized through indicators that can be assessed in practice and used to support effective AI governance \citep{ala2020assessment,kaur2022trustworthy}. For a system assessed under \(P\), let
\[
q_P=(q_{P1},\dots,q_{Pn})\in I_P=[0,1]^n
\]
denote its quantified trustworthiness assessment. Each coordinate \(q_{Pi}\) is the normalized score assigned to dimension \(p_i\) under the measurement, normalization, and aggregation procedures \emph{fixed by} \(P\), and may represent the most recent auditable estimate available rather than an instantaneous measurement. Thus, any numerical agreement in \([0,1]\) is meaningful only relative to the common measurement protocol specified by \(P\): scores obtained through different metrics or assessment procedures are not treated as directly comparable unless \(P\) includes an explicit rule translating them into the same canonical coordinate. Different dimensions may rely on distinct metrics and aggregation rules \citep{rabanser2026towards}; for a broader survey, we refer to \citet{kemmerzell2025towards}. Let 
\begin{align*}
L_P:I_P\to[K],\quad q_P\mapsto L_P(q_P), 
\end{align*} 
be a trustworthiness-level function, where \([K]=\{1,\dots,K\}\). The vector \(q_P\in I_P\) records quantified trustworthiness measurements, while \(L_P(q_P)\) represents the resulting degree of correct functioning through a finite set of levels. \(L_P\) is part of the AI governance datum and maps the quantified assessment \(q_P\) of the trustworthiness profile \(P\) to a finite set of governance-relevant levels. Preferred choices are interpretable and auditable functions, such as stepwise mappings defined by inequalities on the components of \(q_P\): changes within a plateau are treated as tolerable variation, whereas crossing a boundary marks a governance-relevant transition requiring escalation, reassessment, or updated safeguards
\citep{ferrario2025trustworthiness,ferrario_update2026}. Such boundary crossings are treated as governance-relevant transitions requiring reassessment or escalation as they may constitute `substantial modifications' under the EU AI Act \citep{EU_AI_Act_2024,ferrario2026high}; as explained at the end of this section, they also have direct metaphysical consequences.

\begin{definition}[AI system type datum]
\label{def:AI_type}
An \emph{AI system type datum} is a triple
\[
\diam=(F,P,L_P),
\]
where \(F\) is a techno-function, \(P\) is a trustworthiness profile for systems realizing \(F\), and \(L_P\) is the trustworthiness-level function.
\end{definition}

The datum \(\diam\) combines the three elements that determine trustworthiness-based AI system identity  \citep{ferrario2025trustworthiness}. Thus, Ferrario's adaptation of the \texttt{function}\textsuperscript{\tiny +} account developed by \citet{carrara2009fine} to AI systems is similarly governance-oriented. It evaluates AI system identity relative to a fixed techno-function and to the trustworthiness requirements under which that function is correctly realized, with particular emphasis on the quantified assessment of those requirements. Finally, while the measurement and aggregation procedures used to produce the coordinates of \(q_P\) are fixed components of \(P\), and hence of the type datum \(\diam\), \emph{the assessment vector \(q_P\), by contrast, is not part of \(\diam\)}. In fact, it varies across times, deployments, copies, instantiations, and concrete AI system tokens, and will constitute the variable component of the profile-relative states introduced in the remainder of this work. Table~\ref{tab:formal_primitives} summarizes the primitive components of our formalization.

\begin{table}[ht]
\centering
\footnotesize
\renewcommand{\arraystretch}{1.25}
\begin{tabularx}{\textwidth}{
    >{\raggedright\arraybackslash}p{0.20\textwidth}
    >{\raggedright\arraybackslash}X
}
\toprule
\textbf{Symbol} & \textbf{Description} \\
\midrule

\(F\) &
\textbf{Techno-function}: functional capability the AI system is designed to realize, understood as the capability that enables its intended use in a specified domain of use. \\

\(P=\{p_1,\ldots,p_n\}\) &
\textbf{Trustworthiness profile}: the finite set of trustworthiness dimensions, together with their requirements, measurement conventions, evidential conditions, normalization rules, and aggregation procedures. \\

\(q_P\in I_P=[0,1]^n\) &
\textbf{Quantified trustworthiness profile}: the normalized assessment vector produced under the measurement and aggregation procedures fixed by \(P\). \\

\(L_P:I_P\to[K]\) &
\textbf{Trustworthiness-level function}: the governance-level map assigning each quantified profile to a finite trustworthiness level. \\

\(x=(q_P,L_P(q_P))\in S_\diam\) &
\(\diam\)\textbf{-relative AI system state}: a profile-relative state consisting of a quantified trustworthiness profile and its associated trustworthiness level. \\

\(\mathcal U_\diam\) &
\textbf{Set of admissible primitive lifecycle transformations}: the fixed vocabulary of primitive transformations allowed for systems of type \(\diam\), each interpreted as a relation on \(S_\diam\). \\

\bottomrule
\end{tabularx}
\caption{Primitive components of the formalization. }
\label{tab:formal_primitives}
\end{table}

\subsubsection{Trustworthiness-Based AI Identity Criteria}

Fix a type datum \(\diam=(F,P,L_P)\), and let \(a(t)\) and \(b(t)\) denote AI systems of type \(\diam\) considered at time \(t\). Let
\(
q^a_P(t),q^b_P(t)\in I_P
\)
be their quantified trustworthiness profiles, and define
\(
\tau_\diam(a(t)):=L_P(q^a_P(t)), \tau_\diam(b(t)):=L_P(q^b_P(t)).
\)
The AI identity criteria proposed by \citet{ferrario2025trustworthiness} can then be stated as follows.

\begin{definition}[Synchronic and diachronic AI identity]
\label{def:ferrario_identity_criteria}
Let \(a\) and \(b\) be AI systems of the fixed type \(\diam\).

Their \emph{synchronic identity relative to \(\diam\)} at time \(t\) is defined by
\begin{equation}
a(t)=_\diam b(t)
\quad\Longleftrightarrow\quad
\tau_\diam(a(t))=\tau_\diam(b(t)).
\label{eq:ferrario_synchronic_identity}
\end{equation}

The \emph{diachronic identity relative to \(\diam\)} of \(a\) between times \(t_1\) and \(t_2\) is defined by
\begin{equation}
a(t_1)=_\diam a(t_2)
\quad\Longleftrightarrow\quad
\tau_\diam(a(t_1))=\tau_\diam(a(t_2)).
\label{eq:ferrario_diachronic_identity}
\end{equation}
\end{definition}


\begin{figure}[t]
\centering
\tdplotsetmaincoords{20}{30}

\begin{minipage}[t]{0.485\textwidth}
\centering
\vspace{0pt}
\begin{tikzpicture}[
    font=\small,
    >=Stealth,
    line cap=round,
    line join=round
]

\tikzset{
  axis/.style={->, line width=0.8pt},
  curve/.style={line width=1pt},
  guide/.style={densely dashed, line width=0.55pt}
}

\begin{scope}[shift={(0,3.35)}]
  \draw[axis] (0,0) -- (5.9,0);
  \draw[axis] (0,0) -- (0,2.25);

  \node[above] at (2.9,2.05) {accuracy};
  \node[left] at (0,0) {$0$};
  \node[left] at (0,2.0) {$1$};

\draw[curve] plot[smooth,tension=0.65] coordinates {
  (0.15,1.86) (0.80,1.93) (1.40,1.82) (2.30,1.74)
  (3.10,1.63) (3.95,1.45) (4.75,1.55) (5.55,1.68)
};

  \foreach \x/\lab in {1.25/{\(t\)},3.05/{\(t'\)},4.55/{\(t''\)}}{
    \draw[guide] (\x,0) -- (\x,2.05);
    \node[below] at (\x,-0.08) {\lab};
  }
\end{scope}

\begin{scope}[shift={(0,0.35)}]
  \draw[axis] (0,0) -- (5.9,0) node[right] {time};
  \draw[axis] (0,0) -- (0,2.25);

  \node[above] at (2.9,2.05) {robustness};
  \node[left] at (0,0) {$0$};
  \node[left] at (0,2.0) {$1$};

\draw[curve] plot[smooth,tension=0.65] coordinates {
  (0.15,1.30) (0.90,1.78) (1.40,1.82) (2.20,1.55)
  (3.05,1.50) (3.80,1.15) (4.55,1.00) (5.50,0.92)
};

  \foreach \x/\lab in {1.25/{\(t\)},3.05/{\(t'\)},4.55/{\(t''\)}}{
    \draw[guide] (\x,0) -- (\x,2.05);
    \node[below] at (\x,-0.08) {\lab};
  }
\end{scope}

\end{tikzpicture}
\end{minipage}%
\hspace{0.015\textwidth}%
\begin{minipage}[t]{0.485\textwidth}
\centering
\vspace{0pt}
\begin{tikzpicture}[tdplot_main_coords, scale=3.7, line join=round, line cap=round]

\tikzset{
  traj/.style={line width=0.65pt, densely dotted},
  pt/.style={circle, fill, inner sep=1.15pt},
  lab/.style={font=\small},
  axis3d/.style={-{Stealth[length=2.5mm]}, line width=0.9pt}
}

\def\zOne{1.00}
\def\zTwo{1.5}
\def\zThr{2.0}

\coordinate (L1A) at (0.0,0.0,\zOne);
\coordinate (L1B) at (1.0,0.0,\zOne);
\coordinate (L1C) at (1.0,0.4,\zOne);
\coordinate (L1D) at (0.4,0.4,\zOne);
\coordinate (L1E) at (0.4,1.0,\zOne);
\coordinate (L1F) at (0.0,1.0,\zOne);

\coordinate (L2A) at (0.4,0.4,\zTwo);
\coordinate (L2B) at (1.0,0.4,\zTwo);
\coordinate (L2C) at (1.0,0.7,\zTwo);
\coordinate (L2D) at (0.7,0.7,\zTwo);
\coordinate (L2E) at (0.7,1.0,\zTwo);
\coordinate (L2F) at (0.4,1.0,\zTwo);

\coordinate (L3A) at (0.7,0.7,\zThr);
\coordinate (L3B) at (1.0,0.7,\zThr);
\coordinate (L3C) at (1.0,1.0,\zThr);
\coordinate (L3D) at (0.7,1.0,\zThr);

\filldraw[draw, line width=0.55pt, fill=black!25]
  (L1A)--(L1B)--(L1C)--(L1D)--(L1E)--(L1F)--cycle;

\filldraw[draw, line width=0.55pt, fill=black!15]
  (L2A)--(L2B)--(L2C)--(L2D)--(L2E)--(L2F)--cycle;

\filldraw[draw, line width=0.55pt, fill=black!7]
  (L3A)--(L3B)--(L3C)--(L3D)--cycle;

\coordinate (aT)   at (0.92,0.90,\zThr);
\coordinate (aTp)  at (0.82,0.75,\zThr);
\coordinate (aTpp) at (0.76,0.50,\zTwo);

\node[pt] at (aT)   {};
\node[pt] at (aTp)  {};
\node[pt] at (aTpp) {};

\node[lab, anchor=east] at (aT) {$a(t)$};
\node[lab, anchor=west] at (aTp) {$a(t')$};
\node[lab, anchor=south] at (aTpp) {$a(t'')$};

\draw[axis3d] (0,0,\zOne) -- (1.12,0,\zOne) node[anchor=west] {accuracy};
\draw[axis3d] (0,0,\zOne) -- (0,1.12,\zOne) node[anchor=south] {robustness};
\draw[axis3d] (0,0,\zOne) -- (0,0,3.45) node[anchor=east] {\(L_P(q_P)\)};

\node[lab] at (-0.15,0.9,\zOne) {level \(1\)};
\node[lab] at (0.45,1.15,\zTwo) {level \(2\)};
\node[lab] at (0.95,1.15,\zThr) {level \(3\)};

\end{tikzpicture}
\end{minipage}

\caption{Left: two normalized trustworthiness dimensions, here accuracy and robustness, evolve over time under gradual degradation and local improvement. Right: the trustworthiness-level function \(L_P\) maps quantified profiles \(q_P=(q_{P,\mathrm{acc}},q_{P,\mathrm{rob}})\) to trustworthiness levels \(L_P(q_P)\). In this example, level \(3\) is assigned on \([0.7,1]\times[0.7,1]\), level \(2\) on \([0.4,1]\times[0.4,1]\setminus[0.7,1]\times[0.7,1]\), and level \(1\) on the remaining part of \([0,1]\times[0,1]\). The states \(a(t)\) and \(a(t')\) lie on level \(3\), so they satisfy \(\tau_\diam(a(t))=\tau_\diam(a(t'))\). The state \(a(t'')\) lies on level \(2\), so the diachronic identity criterion at \(t\) and \(t''\) or \(t'\) and \(t''\) is not satisfied. This makes visually explicit the trustworthiness comparisons used in Ferrario's synchronic and diachronic identity criteria, prior to the categorical formalization of states and histories.}
\label{fig:time_series_and_levels}
\end{figure}
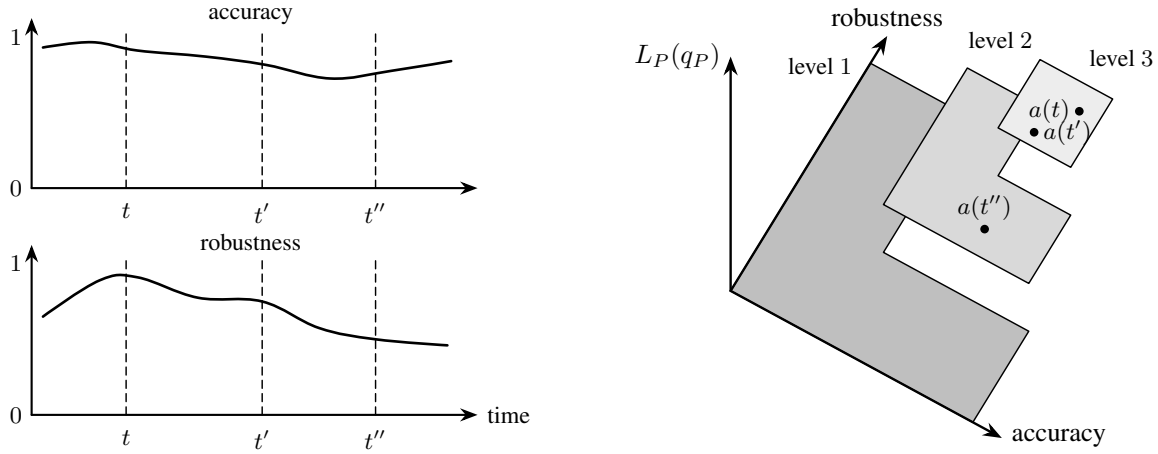

These criteria express identity relative to the abstraction fixed by \(\diam\): fixing \(\diam\) does not determine a unique AI system token and it leaves substantial room for variation within those constraints. Different copies of the same system template may be deployed in distinct but type-compatible contexts, pursue the same techno-function, satisfy the same high-level requirements documented, for example, in a common conformity assessment \citep{EU_AI_Act_2024}, and be evaluated by the same level function \(L_P\), while still receiving different quantified assessments \(q^a_P(t)\) and therefore possibly different trustworthiness levels. They may also differ in model versions, data pipelines, hardware and software configurations, interfaces, and deployment channels---for instance, one copy may be accessed through a web browser and another through a mobile application---and may undergo different MLOps interventions. Despite these changes, these instantiations may remain identical in the sense specified by Definition~\ref{def:ferrario_identity_criteria} if their trustworthiness levels coincide. 

Figure~\ref{fig:time_series_and_levels} illustrates how quantified trustworthiness measurements \(q_P^a(t)\) are mapped by the level function \(L_P\) to trustworthiness levels \(\tau_\diam(a(t))=L_P(q_P^a(t))\), which are the values compared in Ferrario's synchronic and diachronic AI identity criteria in Definition~\ref{def:ferrario_identity_criteria}. At this stage, the metaphysical status of \(a(t)\) and \(b(t)\) remains deliberately underdetermined. These symbols denote AI systems considered at particular times, but the criteria in Definition~\ref{def:ferrario_identity_criteria} do not yet specify what an AI system state is, which transformations may connect such states, how these transformations compose, or how histories of AI system states should be represented over time. Furthermore, the biconditionals in Definition~\ref{def:ferrario_identity_criteria} characterize \(=_\diam\) as a type-relative equivalence relation induced by equality of trustworthiness levels. They do not provide a structural account of the states being related, the admissible transformations connecting them, or the temporal organization of identity-preserving histories.
The category-theoretic construction developed in Sections~\ref{section:cat_AI_identity} and~\ref{section:AI_identity_functors} supplies this missing structure. It defines profile-relative states, admissible lifecycle transformations, trustworthiness-preserving reachability, temporally ordered histories, and time-synchronous comparisons between realized histories. As a result, the propositional criteria in Definition~\ref{def:ferrario_identity_criteria} can be recovered and enriched: the weak interpretation identifies systems through equality of trustworthiness levels, while the strong interpretation refines identity through transformation-grounded reachability and categorical isomorphisms.

\section{The Category \(\Sys_\diam\) of AI System States}
\label{section:cat_AI_identity}

\subsection{Abstract AI System States and Their Transformations}
\label{subsection:abstract_state_transf}
Fix an AI system type datum \(\diam=(F,P,L_P)\). We introduce the last two primitive elements of our formalization: $\diam$-relative AI system states and their transformations. We summarize them in Table~\ref{tab:formal_primitives}.

\begin{definition}[\(\diam\)-relative AI system state]
\label{def:AI_system_state}
A \emph{\(\diam\)-relative AI system state} is an element \(x=(q_P,k)\in S_\diam\), where
\[
S_\diam
=
\{(q_P,L_P(q_P))\mid q_P\in I_P\}
\subseteq I_P\times[K].
\]
For \(x=(q_P,k)\in S_\diam\), the vector \(Q_x:=q_P\in I_P\) is the \emph{quantified trustworthiness profile} of \(x\), and the value \(\tau_\diam(x):=k=L_P(q_P)\) is its \emph{trustworthiness level}.
\end{definition}

By construction, \(S_\diam\) is the graph of the level function \(L_P:I_P\to[K]\) specified as part of the type datum \(\diam\). A state of an AI system, in the sense used here, is therefore a $P$-relative numerical state in \(I_P\times[K]\). It is a governance-relative representation of an AI system, rather than a complete description of its computational, material, organizational, or environmental configuration.
 A single state may be instantiated by a multitude of AI systems: different copies, deployments, model implementations, hardware choices, or software stacks may determine the same vector \(q_P\) and the same level \(L_P(q_P)\). This construction follows from the \texttt{function}\textsuperscript{\tiny +}-trustworthiness account developed in \citet{ferrario2025trustworthiness}: identity is assessed at the level fixed by \(\diam\), rather than by hardware or software sameness. We now introduce transformations between such abstract states.

\begin{definition}[Primitive lifecycle transformation]
Let \(\mathcal U_\diam\) be the set of admissible primitive  lifecycle transformations for AI systems of type \(\diam\). Each \(r\in\mathcal U_\diam\) is associated with a relation \(R_r\subseteq S_\diam\times S_\diam\). For states \(x,y\in S_\diam\), the statement \((x,y)\in R_r\) means that \(y\) may result from \(x\) through an admissible primitive lifecycle transformation recorded as \(r\).
\end{definition}

The set \(\mathcal U_\diam\) is a vocabulary of state transformations that are admissible and primitive. 
Let us discuss them in some detail. First, each \(r\in\mathcal U_\diam\) is interpreted through a relation \(R_r\subseteq S_\diam\times S_\diam.\)
Thus, \((x,y)\in R_r\) means that \(y\) is an admissible possible outcome of applying the primitive lifecycle transformation recorded as \(r\) to a system
in state \(x\). This assertion concerns reachability at the level of abstract states only: it does not imply that \(r\) was actually executed in producing \(y\) from \(x\). Crucially, the associated relations \(R_r\) need \emph{not} preserve trustworthiness levels. Also note that distinct lifecycle paths may connect the same source and target states. For instance, one system may move from \(x\) to \(y\) by a \texttt{model training update}, followed by a \texttt{data validation update}, and then by a \texttt{monitoring configuration update}, while another system may reach the same state through the same three admissible transformations applied in a different order. We show this in Figure~\ref{fig:distinct_paths_same_endpoints}.

Furthermore, \emph{admissibility} of \(r\in\mathcal U_\diam\) is relative to the fixed type datum \(\diam\) and the lifecycle or governance regime under consideration. An intervention included in \(\mathcal U_\diam\) must remain compatible with the techno-function \(F\) and with the assessment framework fixed by \(P\), although it may change the quantified assessment \(q_P\) and may cross a trustworthiness-level boundary. The relation \(R_r\) records only the intervention's possible effect at the level of the abstract state space. It need not be a function: the same intervention may produce different assessment outcomes under different data or deployment conditions, and different interventions may lead to the same abstract state. Examples of elements of \(\mathcal U_\diam\) include \texttt{model retraining}, \texttt{fine-tuning}, \texttt{threshold adjustment}, \texttt{data refresh}, \texttt{model rollback}, and \texttt{documentation update}
\citep{kreuzberger2023machine,eken2025multivocal}. 

Additionally, the elements of \(\mathcal U_\diam\) are \emph{primitive} relative to the chosen granularity. For instance, a transformation such as \texttt{model retraining} may be treated as primitive in a coarse vocabulary \(\mathcal U_\diam\), while a more refined vocabulary \(\mathcal U'_\diam\) may decompose it, for instance, into \texttt{training-data augmentation}, \texttt{hyperparameter tuning}, and \texttt{change of model class}. Fixing \(\mathcal U_\diam\) therefore fixes the level of description at which lifecycle paths are represented in the path category constructed below. A remark on vocabulary refinement is useful before moving to the categorical constructions. If \(\mathcal U_\diam\) is refined, the resulting path categories record more detailed lifecycle provenance, that is, the lifecycle history by which an AI system state is produced, including the particular sequence of updates, interventions, configurations, measurements, and deployment conditions leading to that state. When a refined vocabulary extends a coarser one, the corresponding path categories are related by the inclusion of generators. In what follows, we fix one \(\mathcal U_\diam\) and do not study these refinements further.

\begin{figure}[ht]
\centering

\begin{tikzpicture}[
    scale=0.7,
    transform shape,
    >=Latex,
    state/.style={
        circle,
        draw=black,
        fill=black,
        minimum size=2.6mm,
        inner sep=0pt,
        line width=0.5pt
    },
    pathvertex/.style={
        circle,
        draw=black,
        fill=black,
        minimum size=1.2mm,
        inner sep=0pt,
        line width=0pt
    },
    morph/.style={
        draw=black,
        line width=1.1pt,
        -{Latex[length=2.8mm,width=2mm]},
        line join=round,
        line cap=round,
        shorten <=0pt,
        shorten >=0pt
    },
    labelstyle/.style={
        font=\small
    }
]

\coordinate (x) at (0,0);
\coordinate (y) at (10,0);

\coordinate (u1) at (3,2);   
\coordinate (u2) at (5,1);   

\coordinate (l1) at (2,-1);  
\coordinate (l2) at (7,-2);  

\node[state] (X) at (x) {};
\node[state] (Y) at (y) {};

\node[pathvertex] at (u1) {};
\node[pathvertex] at (u2) {};

\node[pathvertex] at (l1) {};
\node[pathvertex] at (l2) {};

\draw[morph] (x) -- (u1)
    node[pos=0.50, above left =0.5mm, labelstyle, align=left]
    {\texttt{model training}\\\texttt{update}};

\draw[morph] (u1) -- (u2)
    node[pos=0.40, above right=0mm, labelstyle, align=left]
    {\texttt{data validation}\\\texttt{update}};

\draw[morph] (u2) -- (y)
    node[pos=0.50, above=4mm, labelstyle, align=left]
    {\texttt{monitoring}\\\texttt{configuration update}};

\draw[morph] (x) -- (l1)
    node[pos=0.50, below left=1mm, labelstyle, align=left]
    {\texttt{data validation}\\\texttt{update}};

\draw[morph] (l1) -- (l2)
    node[pos=0.50, below=1.5mm, labelstyle, align=left]
    {\texttt{monitoring}\\\texttt{configuration update}};

\draw[morph] (l2) -- (y)
    node[pos=0.50, below right=1mm, labelstyle, align=left]
    {\texttt{model training}\\\texttt{update}};

\node[labelstyle, left=1mm and 0mm of X]
{$(q_P^{1},L_P(q_P^{1}))=x$};

\node[labelstyle, right=1mm and 0mm of Y]
{$y=(q_P^{2},L_P(q_P^{2}))$};

\end{tikzpicture}

\caption{Two distinct lifecycle paths connecting the same abstract states \(x\) and \(y\). The transformations act on different components of the quantified trustworthiness profile \(q_P\), for instance predictive performance, robustness, monitoring, or auditability. In the upper path, the system moves from \(x\) to \(y\) through a \texttt{model training update}, followed by a \texttt{data validation update}, and then a \texttt{monitoring configuration update}. In the lower path, the same three admissible primitive transformations are applied in a different order.}
\label{fig:distinct_paths_same_endpoints}
\end{figure}
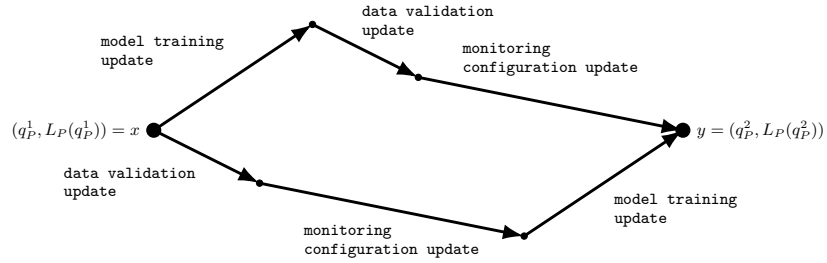

A typical structure for the relations \(R_r\subseteq S_\diam\times S_\diam\) where \(r\in\mathcal U_\diam\) can be given through inequalities on
the quantified trustworthiness profiles. Let \(\pi_j:I_P\to[0,1]\) denote the projection onto the \(j\)-th profile coordinate, and assume that each coordinate is oriented so that larger values represent improvement. For a nonempty set of dimensions \(J\subseteq\{1,\ldots,n\}\), define
\[
R_J
:=
\left\{
(x,y)\in S_\diam\times S_\diam
\;\middle|\;
\pi_j(Q_x)\leq \pi_j(Q_y)
\text{ for every }j\in J
\right\},
\]
where \(R_\varnothing:=S_\diam\times S_\diam\). The coordinates outside \(J\) are left unconstrained. Thus, a primitive transformation \(r\in\mathcal U_\diam\) intended to weakly improve the dimensions in \(J\) may then be represented by \(R_r=R_J\), or more generally by a subrelation \(R_r\subseteq R_J\) when additional preconditions or outcome constraints are required. Relational composition represents the sequential application of primitive transformations. If a transformation constrained by \(J_1\) is followed by
one constrained by \(J_2\), then
\(
R_{J_2}\circ R_{J_1}=R_{J_2\cap J_1}.
\)

\subsection{Categories of Abstract AI System States}
\label{subsection:cat_abstract_AI_states}
Building on the abstract state space \(S_\diam\) and the primitive lifecycle transformations collected in \(\mathcal U_\diam\), we now organize admissible state transformations categorically. As a result, we obtain the thin category \(\Sys_\diam\), equivalently the preorder of trustworthiness-level-preserving reachability on \(S_\diam\). This category
provides the formal basis for the account of AI system identity developed in the remainder of the work.

\begin{definition}[Lifecycle path category]
\label{def:raw_path_category}
Fix a type datum \(\diam=(F,P,L_P)\). The \emph{lifecycle path category} \(\mathsf{Path}_\diam\) is generated as follows. Its objects are the profile-relative states, so
\(
\operatorname{Ob}(\mathsf{Path}_\diam)=S_\diam
\).
For every \(r\in\mathcal U_\diam\) and every pair \(x,y\in S_\diam\) satisfying \((x,y)\in R_r\), introduce a labelled arrow \(r_{x,y}:x\to y\). A morphism \(f:x\rightsquigarrow y\) is a finite path
\[
f
=
\bigl(
x=x_0\xrightarrow{r_1}x_1
\xrightarrow{r_2}\cdots
\xrightarrow{r_k}x_k=y
\bigr),
\]
where \(r_i\in\mathcal U_\diam\) and \((x_{i-1},x_i)\in R_{r_i}\) for every \(i=1,\dots,k\). The identity morphism \(\id_x:x\to x\) is the empty path at \(x\), and composition is concatenation of paths.
\end{definition}

\(\mathsf{Path}_\diam\) is a free path category. Its morphisms are sequences of admissible primitive transformations between abstract states. These transformations are labelled at the granularity level fixed by \(\mathcal U_\diam\). Furthermore, \(\mathsf{Path}_\diam\) preserves the lifecycle provenance represented by the chosen vocabulary as different paths between the same source and target remain distinct. It is abstract and does not assert that every path is realized by a deployed system or that its intermediate states occur in temporal order: considerations on realization and temporal admissibility are imposed only in Section~\ref{section:AI_identity_functors}. As \(\mathsf{Path}_\diam\) is a free path category, it has no non-identity isomorphisms: even when paths
\(f:x\rightsquigarrow y\) and
\(g:y\rightsquigarrow x\)
both exist, their composites are non-empty loops and are not equal to the empty identity paths. Thus, a path \(g\) representing an operational rollback may reconstruct a previous profile-relative state \(x\) without \emph{categorically} inverting the path \(f\) that produced the current state \(y\). In summary, \(\mathsf{Path}_\diam\) is too large for our categorical construction of AI identity. As a next step, let us introduce trustworthiness levels into our construction.

\begin{definition}[Trustworthiness-level-preserving path]
\label{def:tau_preserving_path}
Let
\(
f
=
\bigl(
x=x_0\xrightarrow{r_1}x_1
\xrightarrow{r_2}\cdots
\xrightarrow{r_k}x_k=y
\bigr)
\)
be a morphism in \(\mathsf{Path}_\diam\). The path \(f\) is \emph{trustworthiness-level-preserving} if
\(
\tau_\diam(x_0)
=
\tau_\diam(x_1)
=
\cdots
=
\tau_\diam(x_k).
\)
Equivalently, every state occurring along \(f\) has the level \(\tau_\diam(x)=\tau_\diam(y)\).
\end{definition}

Let \(\mathsf{Path}^{\tau}_\diam\) denote the wide subcategory of
\(\mathsf{Path}_\diam\) having the same objects and only
trustworthiness-level-preserving paths as morphisms. It is a category because
identity paths preserve trustworthiness level and concatenation preserves this
property. This pathwise condition is stronger than equality of endpoint
levels: a path that leaves a level fibre and later returns to it is excluded.

Level preservation does not require the states along a path to be equal. For
\(f=(x_0,\ldots,x_k)\), one may have \(Q_{x_i}\neq Q_{x_j}\), while
\(\tau_\diam(x_i)=\tau_\diam(x_j)\) for all \(i,j\). Thus, quantified
trustworthiness profiles may vary along a path while remaining within one
trustworthiness-level fibre.

\begin{definition}[Parallel-path equivalence]
\label{def:tau_path_equivalence}
For every \(x,y\in S_\diam\), declare all morphisms in
\(\operatorname{Hom}_{\mathsf{Path}^{\tau}_\diam}(x,y)\) equivalent. Thus,
for parallel level-preserving paths \(f,g:x\to y\), write
\[
f\sim_\tau g.
\]
This equivalence abstracts from their lengths, intermediate states,
transformation labels, and order.
\end{definition}

\begin{proposition}
\label{prop:tau_congruence}
The family of relations \(\sim_\tau\) is a categorical congruence on
\(\mathsf{Path}^{\tau}_\diam\).
\end{proposition}

\begin{proof}
Each \(\sim_\tau\) is the universal equivalence relation on the corresponding hom-set. Moreover, composites of equivalent parallel morphisms are again parallel and therefore equivalent.
\end{proof}

The following category is key for this work.

\begin{definition}[AI system state category]
\label{def:AI_state_category}
The \emph{AI system state category} of type \(\diam\) is the quotient category
\[
\Sys_\diam
:=
\mathsf{Path}^{\tau}_\diam/{\sim_\tau}.
\]
Its objects are the states in \(S_\diam\). A morphism from \(x\) to \(y\) in \(\Sys_\diam\) is an equivalence class
\(
[f]_\tau
\in
\operatorname{Hom}_{\Sys_\diam}(x,y)
\)
of trustworthiness-level-preserving  paths \(f:x\rightsquigarrow y\) in \(\mathsf{Path}^{\tau}_\diam\). Identities and composition are given by
\[
\id_x=[\id_x]_\tau,
\qquad
[g]_\tau\circ[f]_\tau
=
[g\circ f]_\tau.
\]
\end{definition}

Let us discuss \(\Sys_\diam\) in some detail. 

\textbf{The thinness of \(\Sys_\diam\), preorders, and trustworthiness-preserving reachability.} Let \(\preceq_\tau\) be the relation on \(S_\diam\):

\begin{definition}[Trustworthiness-preserving reachability preorder]
\label{def:tau_reachability_preorder}
For \(x,y\in S_\diam\), define
\[
x\preceq_\tau y
\quad\Longleftrightarrow\quad
\text{there exists a trustworthiness-level-preserving path }
f:x\rightsquigarrow y
\text{ in }\mathsf{Path}^{\tau}_\diam.
\]
The relation \(\preceq_\tau\) is reflexive, since every identity path preserves its trustworthiness level, and transitive, since the concatenation of two composable trustworthiness-level-preserving paths is again trustworthiness-level-preserving. Hence, \((S_\diam,\preceq_\tau)\) is a preorder.
\end{definition}

\begin{proposition}[Trustworthiness-preserving reachability] \label{prop:Sys_reachability} 
For every \(x,y\in S_\diam\), \( \Hom_{\Sys_\diam}(x,y)\neq\varnothing \Longleftrightarrow x\preceq_\tau y. \) 
Furthermore, whenever this hom-set is nonempty, it contains exactly one morphism. Consequently, \(\Sys_\diam\) is the thin category associated with the trustworthiness-preserving reachability preorder \((S_\diam,\preceq_\tau)\). 
\end{proposition} 
\begin{proof} 
By construction, a morphism from \(x\) to \(y\) in \(\Sys_\diam\) is an equivalence class of trustworthiness-level-preserving paths from \(x\) to \(y\). Hence such a morphism exists exactly when \(x\preceq_\tau y\). Any two parallel level-preserving paths are identified by \(\sim_\tau\), so the resulting morphism is unique. 
\end{proof}

Proposition~\ref{prop:Sys_reachability} shows that the category \(\Sys_\diam\) represents the abstract existence of trustworthiness-level-preserving reachability between profile-relative states in \(S_\diam\). It does not retain the particular sequence of primitive transformations by which one state is reached from another; that information
is represented in \(\mathsf{Path}^{\tau}_\diam\). Instead,
\(\operatorname{Hom}_{\Sys_\diam}(x,y)\) is nonempty precisely when at least one trustworthiness-level-preserving path from \(x\) to \(y\) exists. Consequently, a morphism in \(\Sys_\diam\) does not specify the length of a
representative path, the number or types of transformations occurring in it, their order, their temporal ordering, or any causal relation among them. For instance, suppose that both \(x\xrightarrow{\texttt{retraining}}y\) and
\( x\xrightarrow{\texttt{documentation update}}z \xrightarrow{\texttt{retraining}}y
\)
preserve a common trustworthiness level at every step. These are distinct morphisms in \(\mathsf{Path}^{\tau}_\diam\), since they have different lengths, intermediate states, and transformation labels. In \(\Sys_\diam\), however, they belong to the same equivalence class and give
rise only to the unique morphism
\(\star_{x,y}\in\operatorname{Hom}_{\Sys_\diam}(x,y).\)
Accordingly, the preorder \((S_\diam,\preceq_\tau)\) encodes the abstract notion of trustworthiness-preserving reachability among \(\diam\)-relative states relationally.

\paragraph{Why is \(\Sys_\diam\) a quotient category?}
The quotient construction separates two operations that play different conceptual roles. Passing from \(\mathsf{Path}_\diam\) to \(\mathsf{Path}^{\tau}_\diam\) restricts admissible reachability to paths that remain within one trustworthiness-level fibre. Passing from
\(\mathsf{Path}^{\tau}_\diam\) to \(\Sys_\diam\) then identifies all parallel level-preserving paths, thereby removing their lifecycle provenance and retaining only the existence of trustworthiness-preserving reachability. Equivalently, although in a less explicit way, \(\Sys_\diam\) could have been defined directly as the thin category associated with the preorder \((S_\diam,\preceq_\tau)\) of Definition~\ref{def:tau_reachability_preorder}.  The categories \(\mathsf{Path}_\diam\), \(\mathsf{Path}^{\tau}_\diam\), and \(\Sys_\diam\) are summarized in the upper part of  Table~\ref{tab:categorical_architecture}.

\paragraph{Thinness and isomorphism in \(\Sys_\diam\).}
For \(x,y\in S_\diam\), one has \(x\cong y\) in \(\Sys_\diam\) if and only if \(x\preceq_\tau y\) and \(y\preceq_\tau x\). Indeed, these relations determine unique morphisms \(\star_{x,y}\in\operatorname{Hom}_{\Sys_\diam}(x,y)\) and \(\star_{y,x}\in\operatorname{Hom}_{\Sys_\diam}(y,x)\). Their composites are endomorphisms of \(x\) and \(y\), respectively, and thinness implies \(\star_{y,x}\circ\star_{x,y}=\id_x\) and \(\star_{x,y}\circ\star_{y,x}=\id_y\), since each endomorphism set contains only the identity morphism. Conversely, any isomorphism provides morphisms in both directions and therefore bidirectional trustworthiness-preserving reachability. Such an isomorphism expresses only mutual trustworthiness-level-preserving reachability between the two abstract states, not the reversal of a concrete lifecycle process.

\begin{table}[ht]
\centering
\footnotesize
\renewcommand{\arraystretch}{1.25}
\begin{tabularx}{\textwidth}{
    >{\raggedright\arraybackslash}p{0.17\textwidth}
    >{\raggedright\arraybackslash}X
    >{\raggedright\arraybackslash}X
    >{\raggedright\arraybackslash}X
}
\toprule
\textbf{Category} &
\textbf{Objects} &
\textbf{Morphisms} &
\textbf{Description and role}
\\
\midrule

\(\mathsf{Path}_\diam\) &
All abstract \(\diam\)-relative states \(x=(q_P,L_P(q_P))\). &
Finite paths of primitive admissible lifecycle transformations. &
Retains lifecycle provenance independently of realization and temporal order.
\\

\(\mathsf{Path}^{\tau}_\diam\) &
The same abstract states as \(\mathsf{Path}_\diam\). &
Paths whose states all have one common trustworthiness level. &
Restricts transformations to those compatible with trustworthiness-level invariance.
\\

\(\Sys_\diam
=
\mathsf{Path}^{\tau}_\diam/{\sim_\tau}\) &
All \(\diam\)-relative states. &
Equivalence classes of parallel level-preserving paths; each nonempty hom-set is a singleton. &
Forgets lifecycle provenance and retains abstract trustworthiness-preserving reachability. This reachability is independent of realization, time, and causal order. Equivalently defined as the thin category associated to the preorder \((S_\diam,\preceq_\tau)\).  \\
\midrule
\(\Time_{[t_0,T]}\) &
Times \(t\in[t_0,T]\). &
A unique arrow \(t\to t'\) whenever \(t\leq t'\). &
Provides the shared temporal order for realized histories and same-time comparisons.
\\
\([\Time_{[t_0,T]},\Sys_\diam]\) &
All functors
\(\Time_{[t_0,T]}\to\Sys_\diam\). &
Natural transformations between such functors. &
Provides the ambient category of possible finite histories, whether realized or merely theoretical.
\\
\(\Traj_{\diam,[t_0,T]}\) &
Temporally admissible AI system histories instantiated over $[t_0,T]$ by at least one deployed AI system token. &
Natural transformations whose components admit time-synchronous representatives through states realized at the relevant time. &
Represents realized AI system evolutions and their coherent same-time comparisons.
\\

\bottomrule
\end{tabularx}
\caption{Main categories of the formalization. The construction moves from abstract lifecycle paths and trustworthiness-preserving reachability to time-indexed histories and realized AI system histories.}
\label{tab:categorical_architecture}
\end{table}

\section{AI system histories as Functors: The Category \(\Traj_{\diam,[t_0,T]}\)}
\label{section:AI_identity_functors}

\subsection{Time as a Category}
\label{subsection:time_cat}
As deployed AI system tokens evolve over time, it is necessary to embed time into our categorical approach to AI identity. To do so, we represent time as a poset category as follows.

\begin{definition}[Time category]
Fix a deployment time \(t_0\in\mathbb R\). Let \(\Time_{t_0}\) be the poset category induced by the total order on \([t_0,\infty)\). Its objects are time points, \(\operatorname{Ob}(\Time_{t_0})=[t_0,\infty)\), and its morphisms are given by
\[
\Hom_{\Time_{t_0}}(t,t')=
\begin{cases}
\{\star_{t,t'}\} & \text{if } t\leq t',\\
\varnothing & \text{otherwise.}
\end{cases}
\]
Composition is induced by transitivity of \(\leq\).
\end{definition}

A morphism \(\star_{t,t'}\) expresses that \(t'\) is not earlier than \(t\). \(\Time_{t_0}\) is thin. 

\begin{definition}[Time-shift functor]
Let \(\Time_0\) be the time category with objects \([0,\infty)\). For each deployment time \(t_0\), define the time-shift functor \(s_{t_0}:\Time_0\to\Time_{t_0}\) by \(s_{t_0}(t)=t_0+t\). If \(A:\Time_{t_0}\to\mathcal C\) is a functor, its elapsed-time reparameterization is \(\widetilde A=A\circ s_{t_0}:\Time_0\to\mathcal C\).
\end{definition}

This allows copies or related systems deployed at different calendar times to be compared by time since deployment. In what follows, \(\Time\) denotes either an absolute time category \(\Time_{t_0}\) or the elapsed-time category \(\Time_0\), depending on the comparison at issue.

\subsection{Time-Relative Histories of AI Systems as Functors}
\label{subsection:AI_sys_functor}
The central idea is that each (post-deployment) history of AI systems---understood here as the trajectory through time of AI systems that continue to function correctly, i.e., preserve their trustworthiness level over time---can be represented as a time-order-preserving functor from the time category \(\mathsf T_{t_0}\) to the state category \(\Sys_\diam\). Such histories include, for instance, those of deployed copies of a medical AI system used in a hospital emergency department. This functor assigns a profile-relative state to each time and the unique trustworthiness-preserving reachability morphism to each ordered pair of times. Since \(\Sys_\diam\) is defined independently of temporal order, temporal admissibility additionally requires every assigned morphism to admit at least one representative compatible with the temporal evolution encoded by the functor. Let us elaborate on this construction in a few steps.

\paragraph{From abstract states to time-ordered AI system states and their transformations.}
A morphism in \(\Sys_\diam\) records only the existence of
trustworthiness-level-preserving reachability between two abstract states.
It may be represented by several paths in
\(\mathsf{Path}^{\tau}_\diam\), not all of which need to respect the temporal order of a history. To represent an AI system lifecycle, we therefore require that every morphism assigned to a temporal interval admit at least one representative whose intermediate states occur in the image of the history functor in non-decreasing temporal order.

Let
\(\mathsf A_\diam:\Time_{t_0}\to\Sys_\diam\)
be a functor. For every \(t\geq t_0\), write
\[
\mathsf A_\diam(t)
=
\bigl(
Q_P^{\mathsf A}(t),
\tau_\diam^{\mathsf A}(t)
\bigr)
\in S_\diam,
\]
where
\(
Q_P^{\mathsf A}(t)
=
\bigl(
q_{P1}^{\mathsf A}(t),\ldots,q_{Pn}^{\mathsf A}(t)
\bigr)
\in I_P
\)
is the quantified trustworthiness profile assigned at time \(t\), and
\(
\tau_\diam^{\mathsf A}(t)
=
L_P\bigl(Q_P^{\mathsf A}(t)\bigr)
\)
is its trustworthiness level. For \(t\leq t'\), denote by
\(
\mathsf A_{t,t'}
:=
\mathsf A_\diam(\star_{t,t'})
\in\Hom_{\Sys_\diam}(\mathsf A_\diam(t),\mathsf A_\diam(t'))
\)
the morphism assigned to the unique time arrow \(\star_{t,t'}\in\Hom_{\Time_{t_0}}(t,t')\) by functoriality of \(\mathsf A_\diam\).

\begin{definition}[\(\mathsf A_\diam\)-time-ordered representative]
\label{def:time_ordered_representative}
Let
\(\mathsf A_\diam:\Time_{t_0}\to\Sys_\diam\)
be a functor and let \(t\leq t'\). A path
\[
f=
\bigl(
\mathsf A_\diam(t)=x_0
\xrightarrow{r_1}x_1
\xrightarrow{r_2}\cdots
\xrightarrow{r_m}x_m
=\mathsf A_\diam(t')
\bigr)
\]
in \(\mathsf{Path}^{\tau}_\diam\) representing the morphism
\(\mathsf A_{t,t'}\)
is called an
\emph{\(\mathsf A_\diam\)-time-ordered representative over \([t,t']\)}
if either:
\begin{enumerate}[label=(\alph*),nosep]
    \item \(m=0\) and
    \(\mathsf A_\diam(s)=\mathsf A_\diam(t)\)
    for every \(s\in[t,t']\); or
    \item \(m\geq1\) and there exist times
    \(t=s_0\leq s_1\leq\cdots\leq s_m=t'\)
    such that
    \(x_i=\mathsf A_\diam(s_i)\)
    for every \(i=0,\ldots,m\).
\end{enumerate}
\end{definition}

Clause~(a) represents stationary persistence over a non-degenerate interval: time may pass while the functor assigns the same profile-relative state. In the non-stationary case, every state occurring in the representative path must be assigned by the history at some time in the interval, in an order compatible with that of \(\Time_{t_0}\). Since the representative belongs to \(\mathsf{Path}^{\tau}_\diam\), trustworthiness-level preservation is already built into the construction. That said, time ordering does not turn the relations \(R_{r_i}\) into claims that the corresponding interventions were historically executed: the labels \(r_i\) identify admissible transformations witnessing a temporally coherent path, not certified causal provenance. Time-ordered representatives are stable under concatenation. In fact, if
\(t\leq t'\leq t''\), concatenating time-ordered representatives over \([t,t']\) and \([t',t'']\) yields a time-ordered representative over \([t,t'']\). Functoriality ensures that the resulting path represents the
composite morphism \(\mathsf A_{t',t''}\circ\mathsf A_{t,t'}=\mathsf A_{t,t''}\). We arrive at a key definition for this work.

\begin{definition}[AI system history]
\label{def:AI_history}
Fix an AI system type datum \(\diam=(F,P,L_P)\) and a deployment time \(t_0\). An \emph{AI system history of type \(\diam\)} is a functor
\[
\mathsf A_\diam:\Time_{t_0}\to\Sys_\diam
\]
such that, for every \(t\leq t'\), the morphism
\[
\mathsf A_{t,t'}
\in\Hom_{\Sys_\diam}(\mathsf A_\diam(t),\mathsf A_\diam(t'))
\]
admits at least one \(\mathsf A_\diam\)-time-ordered representative over \([t,t']\).
\end{definition}

By Definition~\ref{def:AI_history}, an AI system history assigns an abstract state to every time and encodes time-ordered, trustworthiness-preserving reachability to every ordered pair of states. We represent them in Figure~\ref{fig:AI_history_equiv_class}.

\textbf{Time-relative states and AI system tokens.}
For each \(t\geq t_0\), the object \(\mathsf A_\diam(t)\) is the \(\diam\)-relative state assigned by the history at time \(t\). It records the quantified trustworthiness profile \(Q_P^{\mathsf A}(t)\) and the corresponding level \(\tau_\diam^{\mathsf A}(t)\). At this stage, the history
need not be instantiated by any concrete AI system token. We will discuss realization by deployed tokens in Section~\ref{subsection:realized_AI_functors}.

\textbf{Morphisms and temporal admissibility.}
For every \(t\leq t'\), the unique morphism \(\mathsf A_{t,t'}\) records trustworthiness-level-preserving reachability between the two states. It may be represented by several paths in \(\mathsf{Path}^{\tau}_\diam\) but, because \(\Sys_\diam\) is defined independently of time, not every such path need respect the temporal order of the history. Definition~\ref{def:AI_history} therefore requires only that the morphism admits at least one \(\mathsf A_\diam\)-time-ordered representative. However, \(\mathsf A_\diam\) does not select any particular lifecycle path; it records morphisms between temporally-ordered AI system states together with the existence of at least one temporally coherent representative.

\textbf{Functoriality and trustworthiness-level invariance.}
The category \(\Time_{t_0}\) contains a unique arrow \(\star_{t,t'}\) for every \(t\leq t'\). Functoriality therefore requires the existence of a morphism \(\mathsf A_{t,t'}\) in \(\Sys_\diam\) for every ordered pair of times. Since every morphism in \(\Sys_\diam\) preserves trustworthiness level, it follows that \(\tau^{\mathsf A}_\diam(t)=\tau^{\mathsf A}_\diam(t')\) for every \(t\leq t'\). Consequently, \(\tau^{\mathsf A}_\diam(t)=\tau^{\mathsf A}_\diam(t')\) for all \(t,t'\geq t_0\), and the entire image of \(\mathsf A_\diam\) lies within a single trustworthiness-level fibre \(L_P^{-1}(k)\times\{k\}\subseteq S_\diam\) for some \(k\in[K]\).
Therefore, states with different trustworthiness levels are objects of \(\Sys_\diam\), but they cannot occur in the image of the same history functor. Indeed, no morphism in \(\Sys_\diam\) can connect them---see Figure~\ref{fig:AI_history_equiv_class}. A history functor therefore represents an uninterrupted trustworthiness-preserving evolution of AI states. A change of trustworthiness level marks the end of one such history and, where appropriate, the beginning of another.

\textbf{Non-uniqueness of histories through a state.}
The same state \(x\in S_\diam\) may occur in the images of many distinct history functors. Such functors may describe different past or future evolutions of the quantified assessment \(Q_P^{\mathsf A}(t)\), and their morphisms may admit different lifecycle representatives. Nevertheless, whenever \(x\) lies in the image of a history functor, the entire image of that functor remains within the trustworthiness-level fibre containing \(x\). Thus, distinct histories may pass through the same state while continuing through different quantified states, provided that all of them retain the same value of \(\tau_\diam\). 

The above considerations show that an AI system history is a time-indexed family of profile-relative states contained within one trustworthiness-level fibre, together with the unique reachability morphisms connecting every temporally ordered pair of those states and the requirement that each such morphism admit a time-ordered representative. In this way, the functor formalizes the persistence criterion proposed by \citet{ferrario2025trustworthiness}: quantified assessments and concrete implementations may change, while the degree of correct functioning encoded by the trustworthiness level remains invariant. We collect these functors in a category before returning to the study of AI identity criteria.

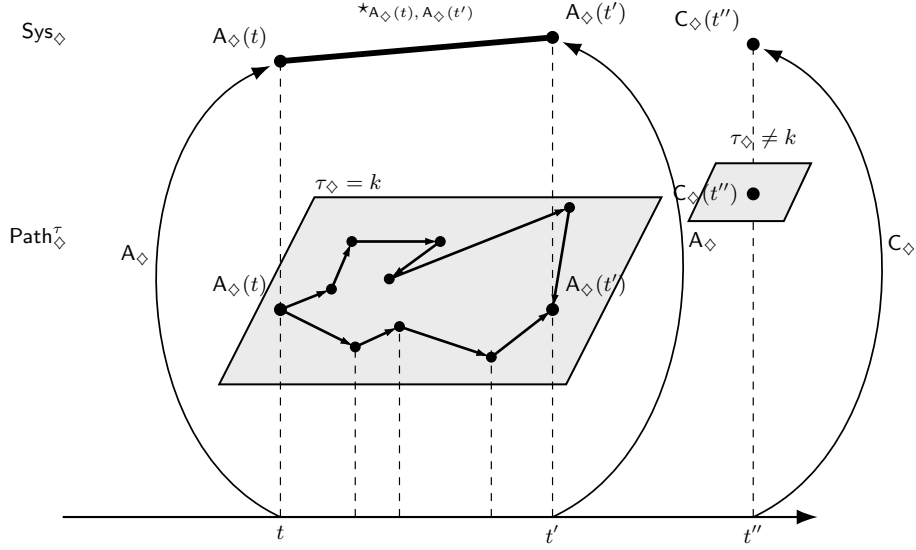
\begin{figure}[ht]
\centering

\begin{tikzpicture}[
    scale=0.90,
    transform shape,
    >=Latex,
    state/.style={
        circle,
        draw=black,
        fill=black,
        minimum size=1.7mm,
        inner sep=0pt,
        line width=0.5pt
    },
    pathvertex/.style={
        circle,
        draw=black,
        fill=black,
        minimum size=1.4mm,
        inner sep=0pt,
        line width=0.5pt
    },
    guide/.style={
        draw=black,
        dashed,
        line width=0.45pt
    },
    timeaxis/.style={
        draw=black,
        line width=0.8pt,
        -{Latex[length=2.8mm,width=2mm]}
    },
    levelplane/.style={
        draw=black,
        fill=black!8,
        line width=0.75pt
    },
    labelstyle/.style={
        font=\small
    },
    assignment/.style={
        draw=black,
        line width=0.65pt,
        -{Latex[length=2.5mm,width=1.8mm]},
        shorten >=4pt
    },
    pathsegA/.style={
        draw=black,
        line width=1.0pt,
        -{Latex[length=2mm,width=1mm]},
        line join=round,
        line cap=round,
        shorten <=0pt,
        shorten >=0pt
    },
    pathsegB/.style={
        draw=black,
        line width=1.0pt,
        -{Latex[length=2mm,width=1mm]},
        line join=round,
        line cap=round,
        shorten <=0pt,
        shorten >=0pt
    },
    sysmorphism/.style={
        draw=black,
        line width=2.2pt,
        line cap=round
    }
]

\coordinate (O) at (0.8,0.4);
\draw[timeaxis] (O) -- (11.9,0.4);

\coordinate (t)   at (4.0,0.4);
\coordinate (t1)  at (5.10,0.4);
\coordinate (t2)  at (5.75,0.4);
\coordinate (t3)  at (7.10,0.4);
\coordinate (tp)  at (8.00,0.4);
\coordinate (tpp) at (10.95,0.4);

\coordinate (aT)   at (4.0,3.45);
\coordinate (aTp)  at (8.0,3.45);
\coordinate (cTpp) at (10.95,5.15);

\coordinate (m1) at (5.10,2.90);
\coordinate (m2) at (5.75,3.20);
\coordinate (m3) at (7.10,2.75);

\coordinate (u1) at (4.75,3.75);
\coordinate (u2) at (5.05,4.45);
\coordinate (u3) at (6.35,4.45);
\coordinate (u4) at (5.60,3.90);
\coordinate (u5) at (8.25,4.95);

\coordinate (PLB) at (3.10,2.35);
\coordinate (PLT) at (4.50,5.10);
\coordinate (PRT) at (9.60,5.10);
\coordinate (PRB) at (8.20,2.35);

\filldraw[levelplane]
    (PLB) -- (PLT) -- (PRT) -- (PRB) -- cycle;

\node[labelstyle] at (5.00,5.30)
    {$\tau_\diam=k$};

\coordinate (aTSys)   at (4.0,7.10);
\coordinate (aTpSys)  at (8.0,7.45);

\coordinate (cTppSys) at ($(cTpp)+(0,2.20)$);

\node[labelstyle,anchor=east] at (1.00,7.50)
    {$\Sys_\diam$};

\node[labelstyle,anchor=east] at (1.00,4.50)
    {$\mathsf{Path}^{\tau}_\diam$};

\draw[guide] (t)   -- (aTSys);
\draw[guide] (tp)  -- (aTpSys);
\draw[guide] (t1)  -- (5.10,2.95);
\draw[guide] (t2)  -- (5.75,3.25);
\draw[guide] (t3)  -- (7.10,2.85);

\draw[guide] (tpp) -- (cTppSys);



\draw[assignment]
    (t)
    to[out=155,in=205]
    node[pos=0.56,left=0.3mm,labelstyle]
    {$\mathsf A_\diam$}
    (aTSys);

\draw[assignment]
    (tp)
    to[out=25,in=-25]
    node[pos=0.56,right=0.3mm,labelstyle]
    {$\mathsf A_\diam$}
    (aTpSys);

\draw[assignment]
    (tpp)
    to[out=25,in=-25]
    node[pos=0.56,right=0.3mm,labelstyle]
    {$\mathsf C_\diam$}
    (cTppSys);


\coordinate (CLB) at ($(cTpp)+(-0.95,-0.40)$);  
\coordinate (CLT) at ($(cTpp)+(-0.55, 0.45)$);  
\coordinate (CRT) at ($(cTpp)+( 0.85, 0.45)$);  
\coordinate (CRB) at ($(cTpp)+( 0.45,-0.40)$);  

\filldraw[levelplane]
    (CLB) -- (CLT) -- (CRT) -- (CRB) -- cycle;

\node[labelstyle,above=1mm of $(CLT)!0.5!(CRT)$]
    {$\tau_\diam\neq k$};

\draw[pathsegA] (aT) -- (m1);
\draw[pathsegA] (m1) -- (m2);
\draw[pathsegA] (m2) -- (m3);
\draw[pathsegA] (m3) -- (aTp);

\draw[pathsegB] (aT) -- (u1);
\draw[pathsegB] (u1) -- (u2);
\draw[pathsegB] (u2) -- (u3);
\draw[pathsegB] (u3) -- (u4);
\draw[pathsegB] (u4) -- (u5);
\draw[pathsegB] (u5) -- (aTp);

\node[pathvertex] at (m1) {};
\node[pathvertex] at (m2) {};
\node[pathvertex] at (m3) {};

\node[pathvertex] at (u1) {};
\node[pathvertex] at (u2) {};
\node[pathvertex] at (u3) {};
\node[pathvertex] at (u4) {};
\node[pathvertex] at (u5) {};

\node[state] (AT)   at (aT)   {};
\node[state] (ATp)  at (aTp)  {};
\node[state] (CTpp) at (cTpp) {};

\node[labelstyle,above left=0mm and 0mm of AT]
    {$\mathsf A_\diam(t)$};

\node[labelstyle,above right=0mm and 0mm of ATp]
    {$\mathsf A_\diam(t')$};

\node[labelstyle, left=0mm and 0mm of CTpp]
    {$\mathsf C_\diam(t'')$};


\draw[sysmorphism]
    (aTSys) -- (aTpSys)
    node[midway,above=2.5mm,labelstyle]
    {$\star_{\mathsf A_\diam(t),\,\mathsf A_\diam(t')}$};

\node[state] (ATSys)   at (aTSys)   {};
\node[state] (ATpSys)  at (aTpSys)  {};
\node[state] (CTppSys) at (cTppSys) {};

\node[labelstyle,above left=0mm and 0mm of ATSys]
    {$\mathsf A_\diam(t)$};

\node[labelstyle,above right=0mm and 0mm of ATpSys]
    {$\mathsf A_\diam(t')$};

\node[labelstyle,above left=0mm and 0mm of CTppSys]
    {$\mathsf C_\diam(t'')$};


\node[labelstyle,below] at (t)   {$t$};
\node[labelstyle,below] at (tp)  {$t'$};
\node[labelstyle,below] at (tpp) {$t''$};

\end{tikzpicture}

\caption{The same trustworthiness-preserving reachability is represented at the path level and in the thin state category. In the lower \(\mathsf{Path}^{\tau}_\diam\) panel, two distinct paths connect \(\mathsf A_\diam(t)\) and \(\mathsf A_\diam(t')\) within the common
trustworthiness-level fibre \(\tau_\diam=k\). The lower path through the states on the displayed time guides is an
\(\mathsf A_\diam\)-time-ordered representative. The upper path is another representative of the same reachability morphism, but its intermediate states are not ordered by the displayed times. The two paths may also differ in
length, intermediate states, and transformation labels.
In the upper \(\Sys_\diam\) panel, the same endpoint states are connected by one thick segment representing the unique element \(\star_{\mathsf A_\diam(t),\mathsf A_\diam(t')}\) of \(\Hom_{\Sys_\diam}
(\mathsf A_\diam(t),\mathsf A_\diam(t'))\).
\(\mathsf C_\diam(t'')\) has \(\Sys_\diam\)-level copy  disconnected from the \(\mathsf A_\diam\)-component because it lies outside the trustworthiness-level fibre \(\tau_\diam=k\).}
\label{fig:AI_history_equiv_class}
\end{figure}

\subsection{The Category \(\Traj_{\diam,[t_0,T]}\)}
\label{subsection:realized_AI_functors}
Finally, we organize the collection of AI system history functors into a trajectory category to study AI identity. Our strategy goes as follows. First, as \(\Sys_\diam\) contains all theoretically possible \(\diam\)-relative states and abstract trustworthiness-level-preserving morphisms, we  restrict our attention to history functors realized by deployed AI system tokens up to a fixed observation time. Then, to later address synchronic AI identity criteria, we impose a synchronicity condition on natural transformations between realized AI system history functors. Fix an observation time \(T\geq t_0\), interpreted as the present time of the analysis. Let \(\Time_{[t_0,T]}\) be the full subcategory of \(\Time_{t_0}\) whose objects are the times \(t\) satisfying \(t_0\leq t\leq T\). In this subsection, an \emph{AI system history over \([t_0,T]\)} means a functor
\(
\mathsf A_\diam:\Time_{[t_0,T]}\to\Sys_\diam
\)
satisfying the temporal-admissibility condition of
Definition~\ref{def:AI_history} for every \(t\leq t'\) in \([t_0,T]\).

\textbf{Realized AI system histories}. 
For a deployed AI system token \(u\) existing throughout
\([t_0,T]\), let \(\sigma_u:[t_0,T]\to S_\diam\)
denote its induced profile-relative state assignment, defined by
\begin{equation*}
\sigma_u(t)
=
\bigl(
q_P^u(t),
L_P(q_P^u(t))
\bigr).
\end{equation*}

We arrive at a key definition:

\begin{definition}[Realized AI system history]
\label{def:realized_AI_history}
An AI system history functor
\(
\mathsf A_\diam:\Time_{[t_0,T]}\to\Sys_\diam
\)
over \([t_0,T]\) is a \emph{realized AI system history} if there exists at least one deployed AI system token \(u\), existing throughout \([t_0,T]\), such that
\begin{equation*}
\sigma_u(t)=\mathsf A_\diam(t)
\qquad
\text{for every }t\in[t_0,T].
\end{equation*}
In that case, \(u\) is said to \emph{instantiate}
\(\mathsf A_\diam\).
\end{definition}

Let \(\mathcal H^{\mathrm{real}}_{\diam,[t_0,T]}\) denote the collection of realized AI system history functors over \([t_0,T]\). Multiple AI system tokens may instantiate the same realized history whenever they occupy the same \(\diam\)-relative state at every time \(t\in[t_0,T]\).\footnote{Realization is understood here in a metaphysical rather than epistemic sense. Whether designers, auditors, or governance actors can determine that such an instantiation exists is a distinct epistemic and practical question. In particular, two deployed tokens may instantiate the same \(\diam\)-relative state only relative to the measurement conventions, normalization rules, and evidential standards fixed by \(P\).}
Since \(\Sys_\diam\) is thin, this common object assignment also determines the morphisms assigned by the functor. For each \(t\in[t_0,T]\), define
\[
S^{\mathrm{real}}_\diam(t)
:=
\left\{
\mathsf C_\diam(t)
\;\middle|\;
\mathsf C_\diam\in
\mathcal H^{\mathrm{real}}_{\diam,[t_0,T]}
\right\}
\subseteq S_\diam.
\]
A state belongs to \(S^{\mathrm{real}}_\diam(t)\) precisely when it is occupied at time \(t\) by at least one deployed token instantiating a realized AI system history.

\textbf{Time-synchronous representatives.}
Let
\(\mathsf A_\diam,\mathsf B_\diam
\in\mathcal H^{\mathrm{real}}_{\diam,[t_0,T]}\).
A categorical comparison between their states at time \(t\) should involve only states that are themselves realized at that same time. Thus, we arrive at:

\begin{definition}[Time-synchronous representative]
\label{def:time_synchronous_representative}
Let \(t\in[t_0,T]\), and let
\(h\in\Hom_{\Sys_\diam}
\bigl(\mathsf A_\diam(t),\mathsf B_\diam(t)\bigr).
\)
A path
\begin{equation*}
f=
\bigl(
\mathsf A_\diam(t)=x_0
\xrightarrow{r_1}x_1
\xrightarrow{r_2}\cdots
\xrightarrow{r_m}x_m
=\mathsf B_\diam(t)
\bigr)
\end{equation*}
in \(\mathsf{Path}^{\tau}_\diam\) representing \(h\) is called
\emph{time-synchronous at \(t\)} if
\begin{equation*}
x_i\in S^{\mathrm{real}}_\diam(t)
\qquad
\text{for every }i=0,\ldots,m.
\end{equation*}
\end{definition}

A time-synchronous representative may pass through states realized by AI system histories other than \(\mathsf A_\diam\) and \(\mathsf B_\diam\), but every intermediate state must be occupied at time \(t\) by at least one deployed token. It therefore defines a \emph{vertical} comparison between realized AI system histories. Importantly, the existence of such a representative does not mean that \(\mathsf B_\diam(t)\) is realized by applying a sequence of instantaneous transformations to \(\mathsf A_\diam(t)\). In fact, the representative is \emph{not} a realized lifecycle trajectory, but an abstract comparison path in \(\mathsf{Path}^{\tau}_\diam\) whose intermediate states are realized at the same time \(t\). Thus, time-synchronous representatives witness theoretical comparability between states occupied at the same time, not synchronic production of one state from another. We will show the existence of such representatives in some examples in Section~\ref{subsec:examples}.

The restriction to \(S^{\mathrm{real}}_\diam(t)\) prevents time-synchronous comparison between the states \(\mathsf  A_\diam(t)\) and \(\mathsf  B_\diam(t)\) from being mediated by merely possible states. Without this restriction, two realized histories \(\mathsf  A_\diam\) and \(\mathsf  B_\diam\) could be compared at time \(t\) through intermediate states that are admissible in the abstract state space \(S_\diam\), but not occupied by any deployed token of type \(\diam\) at that time. The resulting comparison would then be grounded in the abstract structure of \(\Sys_\diam\), rather than in the population of co-realized states. This would oversimplify synchronic comparison: it would treat mathematically admissible bridges through the abstract profile space as if they were available for comparing actually deployed systems. Requiring all states in a time-synchronous representative to lie in \(S^{\mathrm{real}}_\diam(t)\) ensures that vertical comparison remains a relation among states actually realized at the time of comparison, rather than a relation mediated by unoccupied points of the abstract state space.\footnote{For example, suppose \(P\) has four quantified dimensions and \(q_P\in[0,1]^4\). Then a tuple such as
\[
q_P=(0.813,0.742,0.901,0.615)
\]
determines an abstract state
\(
x=((0.813,0.742,0.901,0.615),L_P(0.813,0.742,0.901,0.615))
\in S_\diam.
\)
But it need not be the case that any deployed AI system token in the population under analysis occupies exactly this \(\diam\)-relative state at a given time \(t\). Cardinality considerations aside, the point is both practical and conceptual: even when quantified profiles are represented by floating-point values, the abstract state space \(S_\diam\) is too large, containing many admissible states that are not realized by any token at a given time.}

\vspace{1em}

We are now in the position to introduce the last category of this work. 

\begin{definition}[Category of realized AI system histories]
\label{def:AI_cat}
Fix the type datum \(\diam=(F,P,L_P)\) and an observation interval
\([t_0,T]\). The \emph{category of realized AI system histories of type
\(\diam\)}, denoted by
\[
\Traj_{\diam,[t_0,T]},
\]
has object collection
\[
\operatorname{Ob}\bigl(\Traj_{\diam,[t_0,T]}\bigr)
=
\mathcal H^{\mathrm{real}}_{\diam,[t_0,T]}.
\]

For realized histories
\(\mathsf A_\diam,\mathsf B_\diam\), define
\(
\Hom_{\Traj_{\diam,[t_0,T]}}
\bigl(\mathsf A_\diam,\mathsf B_\diam\bigr)
\)
to be the collection of natural transformations
\[
\eta:
\mathsf A_\diam\Rightarrow\mathsf B_\diam
\]
in
\(
[\Time_{[t_0,T]},\Sys_\diam]
\)
such that every component
\(
\eta_t\in\Hom_{\Sys_\diam}(\mathsf A_\diam(t),\mathsf B_\diam(t))
\)
admits at least one time-synchronous representative at \(t\). Such natural transformations are called \emph{time-synchronous}.
\end{definition}
Since a morphism in \(\Traj_{\diam,[t_0,T]}\) is a natural transformation, its components satisfy, for every \(t\leq t'\),
\begin{equation}
\mathsf B_{t,t'}\circ\eta_t
=
\eta_{t'}\circ\mathsf A_{t,t'}.
\label{eq:naturality}
\end{equation}
Because \(\Sys_\diam\) is thin, this equation holds automatically whenever all four morphisms exist.

As with the preceding categorical constructions, the objects of \(\Traj_{\diam,[t_0,T]}\) are realized AI system histories rather than individual AI system tokens. Distinct deployed tokens are represented by the same object whenever they instantiate the same time-indexed sequence of \(\diam\)-relative states. The morphisms compare realized histories pointwise through states realized at the same time. The following result establishes that these objects and morphisms form a category.

\begin{proposition}
\label{prop:Traj_category}
The category \(\Traj_{\diam,[t_0,T]}\) is a thin subcategory of \([\Time_{[t_0,T]},\Sys_\diam].\)
\end{proposition}

\begin{proof}
For every realized history \(\mathsf A_\diam\), the component of its identity natural transformation at time \(t\) is
\(
\id_{\mathsf A_\diam(t)}.
\)
It is represented by the empty path at \(\mathsf A_\diam(t)\), which is
time-synchronous because
\(
\mathsf A_\diam(t)\in S^{\mathrm{real}}_\diam(t).
\)

Now let
\[
\eta:\mathsf A_\diam\Rightarrow\mathsf B_\diam,
\qquad
\theta:\mathsf B_\diam\Rightarrow\mathsf C_\diam
\]
be time-synchronous natural transformations. For each \(t\), concatenate a time-synchronous representative of \(\eta_t\) with one of \(\theta_t\). Every state of the resulting path belongs to \(S^{\mathrm{real}}_\diam(t)\), so it is a time-synchronous representative of
\(
(\theta\circ\eta)_t
=
\theta_t\circ\eta_t.
\)
Hence time-synchronous natural transformations are closed under identities and composition. Finally, \(\Sys_\diam\) is thin, so between two fixed functors there is at most one natural transformation. Therefore, \(\Traj_{\diam,[t_0,T]}\) is thin.
\end{proof}

\textbf{Morphisms and isomorphisms in \(\Traj_{\diam,[t_0,T]}\).}
The morphisms \(\mathsf A_{t,t'}\) and \(\mathsf B_{t,t'}\) encode the forward trustworthiness-preserving evolution of two realized histories between times \(t\) and \(t'\). A component \(\eta_t\in
\Hom_{\Sys_\diam}
(\mathsf A_\diam(t),\mathsf B_\diam(t))\)
compares their states at the same time \(t\). Equation~\eqref{eq:naturality} expresses the compatibility of these same-time comparisons with temporal evolution. 
Both composites in equation~\eqref{eq:naturality} belong to \(\Hom_{\Sys_\diam}(\mathsf A_\diam(t),\mathsf B_\diam(t'))\).
Since \(\Sys_\diam\) is thin, they necessarily coincide. 

By definition, a time-synchronous representative of \(\eta_t\) contains only states in \(S^{\mathrm{real}}_\diam(t)\). It therefore introduces no
state indexed by a different time into the comparison at \(t\). Any stronger requirement that the assessments \(Q_P^{\mathsf A}(t)\) and \(Q_P^{\mathsf B}(t)\) use only evidence available by time \(t\) must be incorporated into the evidential and measurement conditions fixed by the
trustworthiness profile \(P\). A morphism in \(\Traj_{\diam,[t_0,T]}\) is thus a coherent family of time-synchronous comparisons between realized AI system histories.
Coherence follows automatically from the thinness of \(\Sys_\diam\). Since \(\Traj_{\diam,[t_0,T]}\) is itself thin, two realized histories are isomorphic precisely when morphisms exist in both directions. Thus, \(\mathsf A_\diam
\cong_{\Traj_{\diam,[t_0,T]}} \mathsf B_\diam\)
if and only if \(\Hom_{\Traj_{\diam,[t_0,T]}}
(\mathsf A_\diam,\mathsf B_\diam)\neq\varnothing\)
and \(\Hom_{\Traj_{\diam,[t_0,T]}}
(\mathsf B_\diam,\mathsf A_\diam)\neq\varnothing\).
Equivalently, for every \(t\in[t_0,T]\), the corresponding comparison morphisms in both directions must exist in \(\Sys_\diam\) and admit time-synchronous representatives. Figure~\ref{fig:trustworthiness_level_drop}
illustrates these comparisons and their relation to
trustworthiness-level invariance.

\textbf{Categorical comparability requires a shared type.}
Realized AI system histories with different type data are not directly comparable within the same trajectory category. If
\(\diam=(F,P,L_P)\) and
\(\diam'=(F',P',L_{P'})\)
differ, then
\[
\mathsf A_\diam
\in\operatorname{Ob}\bigl(\Traj_{\diam,[t_0,T]}\bigr)
\qquad\text{and}\qquad
\mathsf B_{\diam'}
\in\operatorname{Ob}\bigl(\Traj_{\diam',[t_0,T]}\bigr)
\]
belong to different categories, and there is no default morphism between them. Philosophically, this expresses the claim that identity comparisons presuppose a shared techno-function, trustworthiness profile, and level function \citep{ferrario2025trustworthiness}. For example, an ICU-triage system and a consumer-credit system may both qualify as AI systems, but they do not share the type datum required for a direct identity comparison in this framework.

Table~\ref{tab:categorical_architecture} presents a summary of the time-relative categories \(\mathsf T_{t_0}\), \([\Time_{[t_0,T]},\Sys_\diam]\), and \(\Traj_{\diam,[t_0,T]}\) presented in this section.

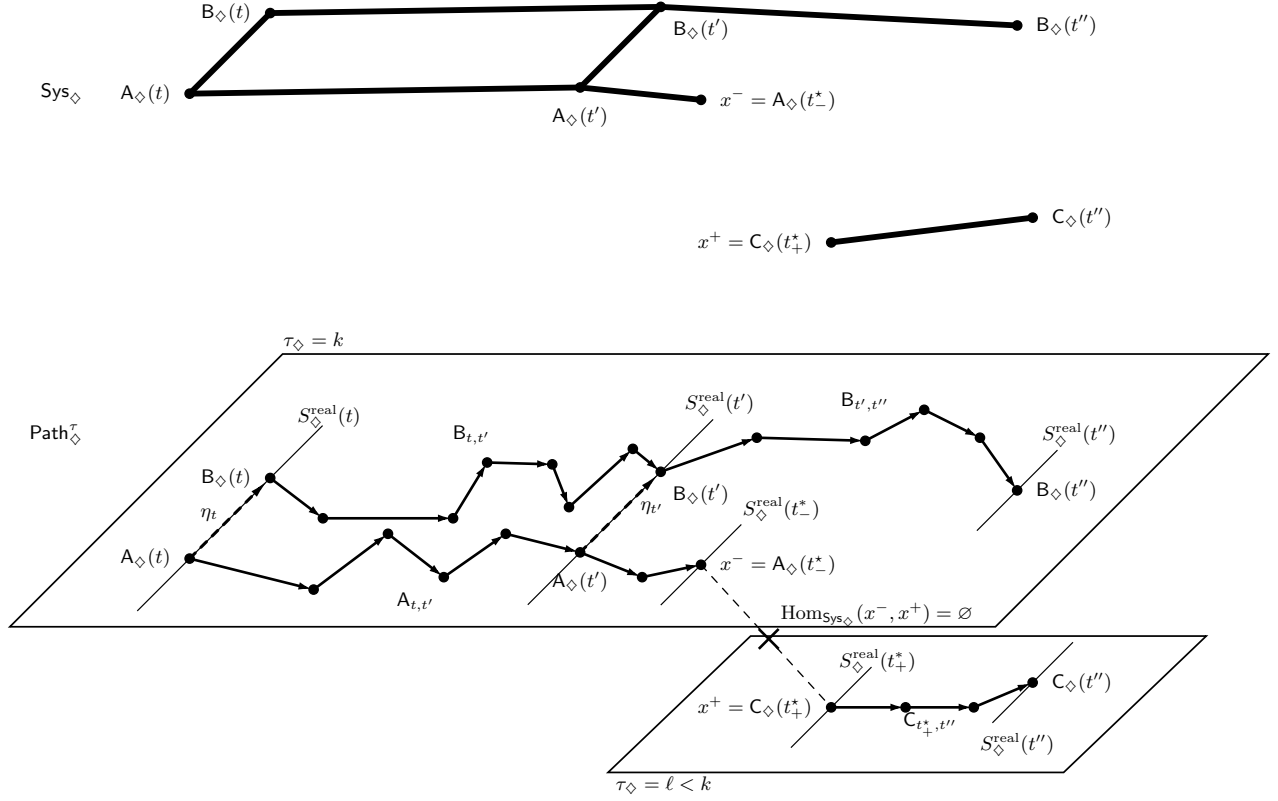
\begin{figure}[ht]
\centering

\begin{tikzpicture}[
    scale=0.82,
    transform shape,
    >=Latex,
    state/.style={
        circle,
        draw=black,
        fill=black,
        minimum size=1.5mm,
        inner sep=0pt,
        line width=0.5pt
    },
    pathvertex/.style={
        circle,
        draw=black,
        fill=black,
        minimum size=1.5mm,
        inner sep=0pt,
        line width=0.5pt
    },
    guide/.style={
        draw=black,
        line width=0.45pt
    },
    plane/.style={
        draw=black,
        line width=0.65pt
    },
    ntcomp/.style={
        draw=black,
        dashed,
        line width=1.15pt,
        -{Latex[length=3mm,width=1.0mm]},
        line join=round,
        line cap=round,
        shorten <=0pt,
        shorten >=0pt
    },
    histpathA/.style={
        draw=black,
        line width=1.0pt,
        -{Latex[length=2mm,width=1mm]},
        line join=round,
        line cap=round,
        shorten <=0pt,
        shorten >=0pt
    },
    histpathB/.style={
        draw=black,
        line width=1.0pt,
        -{Latex[length=2mm,width=1mm]},
        line join=round,
        line cap=round,
        shorten <=0pt,
        shorten >=0pt
    },
    histpathC/.style={
        draw=black,
        line width=1.0pt,
        -{Latex[length=2mm,width=1mm]},
        line join=round,
        line cap=round,
        shorten <=0pt,
        shorten >=0pt
    },
    blocked/.style={
        draw=black,
        dashed,
        line width=0.5pt,
        line cap=round
    },
    labelstyle/.style={
        font=\small
    }
]



\coordinate (P1) at (-1.20,1.55);
\coordinate (P2) at (3.20,5.95);
\coordinate (P3) at (19.10,5.95);
\coordinate (P4) at (14.70,1.55);

\draw[plane] (P1) -- (P2) -- (P3) -- (P4) -- cycle;

\coordinate (Q1) at (8.45,-0.80);
\coordinate (Q2) at (10.75,1.40);
\coordinate (Q3) at (18.10,1.40);
\coordinate (Q4) at (15.80,-0.80);

\draw[plane] (Q1) -- (Q2) -- (Q3) -- (Q4) -- cycle;

\node[labelstyle] at (3.70,6.15)
    {$\tau_\diam=k$};

\node[labelstyle] at (09.35,-1.00)
    {$\tau_\diam=\ell<k$};


\coordinate (aT) at (1.70,2.65);
\coordinate (bT) at (3.00,3.95);

\coordinate (aTp) at (8.00,2.75);
\coordinate (bTp) at (9.30,4.05);

\coordinate (Lt1) at ($(aT)!-0.65!(bT)$);
\coordinate (Lt2) at ($(aT)!1.65!(bT)$);

\coordinate (Rt1) at ($(aTp)!-0.65!(bTp)$);
\coordinate (Rt2) at ($(aTp)!1.65!(bTp)$);

\draw[guide] (Lt1) -- (Lt2);
\draw[guide] (Rt1) -- (Rt2);

\draw[ntcomp]
    (aT) --
    node[pos=0.55,left=1.2mm,labelstyle] {$\eta_t$}
    (bT);

\draw[ntcomp]
    (aTp) --
    node[pos=0.55,right=1.2mm,labelstyle] {$\eta_{t'}$}
    (bTp);


\coordinate (A1) at (3.70,2.15);
\coordinate (A2) at (4.90,3.05);
\coordinate (A3) at (5.80,2.35);
\coordinate (A4) at (6.80,3.05);

\draw[histpathA] (aT)  -- (A1);
\draw[histpathA] (A1)  -- (A2);
\draw[histpathA] (A2)  -- (A3);
\draw[histpathA] (A3)  -- (A4);
\draw[histpathA] (A4)  -- (aTp);


\coordinate (B1) at (3.85,3.30);
\coordinate (B2) at (5.95,3.30);
\coordinate (B3) at (6.50,4.20);
\coordinate (B4) at (7.55,4.17);
\coordinate (B5) at (7.82,3.48);
\coordinate (B6) at (8.85,4.42);

\draw[histpathB] (bT)  -- (B1);
\draw[histpathB] (B1)  -- (B2);
\draw[histpathB] (B2)  -- (B3);
\draw[histpathB] (B3)  -- (B4);
\draw[histpathB] (B4)  -- (B5);
\draw[histpathB] (B5)  -- (B6);
\draw[histpathB] (B6)  -- (bTp);


\coordinate (B7)   at (10.85,4.60);
\coordinate (B8)   at (12.60,4.55);
\coordinate (B9)   at (13.55,5.05);
\coordinate (B10)  at (14.45,4.60);
\coordinate (bTpp) at (15.05,3.75);

\draw[histpathB] (bTp) -- (B7);
\draw[histpathB] (B7)  -- (B8);
\draw[histpathB] (B8)  -- (B9);
\draw[histpathB] (B9)  -- (B10);
\draw[histpathB] (B10) -- (bTpp);

\coordinate (UB1) at ($(bTpp)+(-0.65,-0.65)$);
\coordinate (UB2) at ($(bTpp)+(0.65,0.65)$);

\draw[guide] (UB1) -- (UB2);


\coordinate (A5)    at (9.00,2.35);
\coordinate (aDrop) at (9.95,2.55);

\draw[histpathA] (aTp) -- (A5);
\draw[histpathA] (A5)  -- (aDrop);

\coordinate (DA1) at ($(aDrop)+(-0.65,-0.65)$);
\coordinate (DA2) at ($(aDrop)+(0.65,0.65)$);

\draw[guide] (DA1) -- (DA2);


\coordinate (cStart) at (12.05,0.25);
\coordinate (C1)     at (13.25,0.25);
\coordinate (C2)     at (14.35,0.25);
\coordinate (cEnd)   at (15.30,0.65);

\draw[histpathC] (cStart) -- (C1);
\draw[histpathC] (C1)     -- (C2);
\draw[histpathC] (C2)     -- (cEnd);

\coordinate (LC1) at ($(cStart)+(-0.65,-0.65)$);
\coordinate (LC2) at ($(cStart)+(0.65,0.65)$);

\coordinate (RC1) at ($(cEnd)+(-0.65,-0.65)$);
\coordinate (RC2) at ($(cEnd)+(0.65,0.65)$);

\draw[guide] (LC1) -- (LC2);
\draw[guide] (RC1) -- (RC2);


\draw[blocked]
    (aDrop) --
    node[pos=0.53,right=2mm,labelstyle,align=left]
    {}
    (cStart);

\coordinate (X) at ($(aDrop)!0.52!(cStart)$);

\draw[line width=1.0pt]
    ($(X)+(-0.16,-0.16)$) --
    ($(X)+(0.16,0.16)$);

\draw[line width=1.0pt]
    ($(X)+(-0.16,0.16)$) --
    ($(X)+(0.16,-0.16)$);

\node[labelstyle, above right=1mm of X]
    {$\Hom_{\Sys_\diam}(x^{-},x^{+})=\varnothing$};


\node[pathvertex] at (A1) {};
\node[pathvertex] at (A2) {};
\node[pathvertex] at (A3) {};
\node[pathvertex] at (A4) {};
\node[pathvertex] at (A5) {};

\node[pathvertex] at (B1) {};
\node[pathvertex] at (B2) {};
\node[pathvertex] at (B3) {};
\node[pathvertex] at (B4) {};
\node[pathvertex] at (B5) {};
\node[pathvertex] at (B6) {};
\node[pathvertex] at (B7) {};
\node[pathvertex] at (B8) {};
\node[pathvertex] at (B9) {};
\node[pathvertex] at (B10) {};

\node[pathvertex] at (C1) {};
\node[pathvertex] at (C2) {};


\node[state] (AT)    at (aT)    {};
\node[state] (BT)    at (bT)    {};
\node[state] (ATp)   at (aTp)   {};
\node[state] (BTp)   at (bTp)   {};
\node[state] (ADrop) at (aDrop) {};
\node[state] (BTpp)  at (bTpp)  {};
\node[state] (CStart) at (cStart) {};
\node[state] (CEnd)   at (cEnd)   {};


\node[labelstyle,left=1mm of AT]
    {$\mathsf A_\diam(t)$};

\node[labelstyle,left=1mm of BT]
    {$\mathsf B_\diam(t)$};

\node[labelstyle, below=1mm of ATp]
    {$\mathsf A_\diam(t')$};

\node[labelstyle, below right=0.2mm of BTp]
    {$\mathsf B_\diam(t')$};

\node[labelstyle, right=1mm of BTpp]
    {$\mathsf B_\diam(t'')$};

\node[labelstyle,right=1mm of ADrop]
    {$x^{-}=\mathsf A_\diam(t^\star_-)$};

\node[labelstyle, left=1mm of CStart]
    {$x^{+}=\mathsf C_\diam(t^\star_+)$};

\node[labelstyle,right=1mm of CEnd]
    {$\mathsf C_\diam(t'')$};

\node[labelstyle] at (3.95,4.95)
    {$S^{\mathrm{real}}_\diam(t)$};

\node[labelstyle] at (10.25,5.15)
    {$S^{\mathrm{real}}_\diam(t')$};

\node[labelstyle] at (16.05,4.65)
    {$S^{\mathrm{real}}_\diam(t'')$};

\node[labelstyle] at (15.05,-0.35)
    {$S^{\mathrm{real}}_\diam(t'')$};

\node[labelstyle] at (11.25,3.45)
    {$S^{\mathrm{real}}_\diam(t^{*}_{-})$};

\node[labelstyle] at (12.80,1.0)
    {$S^{\mathrm{real}}_\diam(t^{*}_{+})$};

\node[labelstyle] at (5.35,1.95)
    {$\mathsf A_{t,t'}$};

\node[labelstyle] at (6.25,4.65)
    {$\mathsf B_{t,t'}$};

\node[labelstyle] at (12.60,5.18)
    {$\mathsf B_{t',t''}$};

\node[labelstyle] at (13.65,-0.03)
    {$\mathsf C_{t^\star_+,t''}$};



\tikzset{
    sysedge/.style={
        draw=black,
        line width=2.2pt,
        line cap=round
    }
}


\coordinate (aTsys)     at ($(aT)+(0,7.50)$);
\coordinate (bTsys)     at ($(bT)+(0,7.50)$);
\coordinate (aTpsys)    at ($(aTp)+(0,7.50)$);
\coordinate (bTpsys)    at ($(bTp)+(0,7.50)$);
\coordinate (aDropsys)  at ($(aDrop)+(0,7.50)$);
\coordinate (bTppsys)   at ($(bTpp)+(0,7.50)$);
\coordinate (cStartsys) at ($(cStart)+(0,7.50)$);
\coordinate (cEndsys)   at ($(cEnd)+(0,7.50)$);

\node[labelstyle,anchor=east]
    at ($(aT)+(-1.60,2)$)
    {$\mathsf{Path}^{\tau}_\diam$};

\node[labelstyle,anchor=east]
    at ($(aTsys)+(-1.60,0)$)
    {$\Sys_\diam$};


\draw[sysedge]
    (aTsys) -- (bTsys);

\draw[sysedge]
    (aTsys) -- (aTpsys);

\draw[sysedge]
    (bTsys) -- (bTpsys);

\draw[sysedge]
    (aTpsys) -- (bTpsys);

\draw[sysedge]
    (aTpsys) -- (aDropsys);

\draw[sysedge]
    (bTpsys) -- (bTppsys);

\draw[sysedge]
    (cStartsys) -- (cEndsys);



\node[state] (ATSys)     at (aTsys)     {};
\node[state] (BTSys)     at (bTsys)     {};
\node[state] (ATpSys)    at (aTpsys)    {};
\node[state] (BTpSys)    at (bTpsys)    {};
\node[state] (ADropSys)  at (aDropsys)  {};
\node[state] (BTppSys)   at (bTppsys)   {};
\node[state] (CStartSys) at (cStartsys) {};
\node[state] (CEndSys)   at (cEndsys)   {};


\node[labelstyle,left=1mm of ATSys]
    {$\mathsf A_\diam(t)$};

\node[labelstyle,left=1mm of BTSys]
    {$\mathsf B_\diam(t)$};

\node[labelstyle,below=1mm of ATpSys]
    {$\mathsf A_\diam(t')$};

\node[labelstyle,below right=0.2mm of BTpSys]
    {$\mathsf B_\diam(t')$};

\node[labelstyle,right=1mm of ADropSys]
    {$x^{-}=\mathsf A_\diam(t^\star_-)$};

\node[labelstyle,right=1mm of BTppSys]
    {$\mathsf B_\diam(t'')$};

\node[labelstyle,left=1mm of CStartSys]
    {$x^{+}=\mathsf C_\diam(t^\star_+)$};

\node[labelstyle,right=1mm of CEndSys]
    {$\mathsf C_\diam(t'')$};
\end{tikzpicture}
\caption{The figure separates the path-level and state-category representations of realized AI system histories. In the lower \(\mathsf{Path}^{\tau}_\diam\) layer, the trustworthiness-level planes \(\tau_\diam=k\) and \(\tau_\diam=\ell<k\) contain explicit path representatives. The solid polygonal paths are time-ordered representatives
of the history morphisms \(\mathsf A_{t,t'}\), \(\mathsf B_{t,t'}\), \(\mathsf B_{t',t''}\), and \(\mathsf C_{t^\star_+,t''}\), while the dashed paths \(\eta_t\) and \(\eta_{t'}\) are time-synchronous representatives of the
components of a natural transformation
\(\eta:\mathsf A_\diam\Rightarrow\mathsf B_\diam\).
The slanted lines indicate the co-realized state sets
\(S^{\mathrm{real}}_\diam(t)\),
\(S^{\mathrm{real}}_\diam(t')\),
\(S^{\mathrm{real}}_\diam(t^\star_-)\),
\(S^{\mathrm{real}}_\diam(t^\star_+)\), and
\(S^{\mathrm{real}}_\diam(t'')\). In the upper \(\Sys_\diam\) layer, the same states are connected by  segments whenever the lower layer contains a trustworthiness-level-preserving time-ordered representative between them. Since \(\Sys_\diam\) is thin, each such segment represents the unique morphism between its endpoints, and the naturality condition for the comparison of \(\mathsf A_\diam\) and \(\mathsf B_\diam\) holds. After \(\mathsf A_\diam\) reaches
\(x^-=\mathsf A_\diam(t^\star_-)\) at level \(k\), the post-drop state \(x^+=\mathsf C_\diam(t^\star_+)\) lies at level
\(\ell<k\). Hence \(\Hom_{\Sys_\diam}(x^-,x^+)=\varnothing\), represented in the upper layer
by the absence of a connecting segment. The level drop therefore interrupts the history of \(\mathsf A_\diam\); if the post-drop evolution is realized, it is represented by a new history object beginning at \(t^\star_+\), in
\(\Traj_{\diam,[t^\star_+,T]}\).}
\label{fig:trustworthiness_level_drop}
\end{figure}

\section{AI Identity in Category Theory}
\label{section:AI_identity_category}
Finally, we identify the notions of AI system identity that emerge from the categorical formalization and relate them to the AI identity criteria in Definition~\ref{def:ferrario_identity_criteria}. Let
\(\Traj_{\diam,[t_0,T]}\)
be the category of realized AI system histories of fixed type \(\diam\) introduced in Definition~\ref{def:AI_cat}.

\subsection{Equality of AI System History Functors in \(\Traj_{\diam,[t_0,T]}\)}
We start our investigation of AI identity in \(\Traj_{\diam,[t_0,T]}\) with a very strict concept: equality of functors.

\begin{proposition}
\label{prop:equality_same_objects}
Let
\(\mathsf A_\diam,\mathsf B_\diam\)
be objects of
\(\Traj_{\diam,[t_0,T]}\).
If
\[
\mathsf A_\diam(t)=\mathsf B_\diam(t)
\]
in \(\Sys_\diam\) for every \(t\in[t_0,T]\), then
\[
\mathsf A_\diam=\mathsf B_\diam
\]
as functors, and hence as objects of
\(\Traj_{\diam,[t_0,T]}\).
\end{proposition}

\begin{proof}
For every \(t\leq t'\), both
\(\mathsf A_{t,t'}\) and \(\mathsf B_{t,t'}\)
are morphisms from the common state
\(\mathsf A_\diam(t)=\mathsf B_\diam(t)\)
to the common state
\(\mathsf A_\diam(t')=\mathsf B_\diam(t')\).
Since \(\Sys_\diam\) is thin, these parallel morphisms are equal. The two functors therefore agree on objects and morphisms.
\end{proof}

The functor equality is very strict and has limited usefulness for AI identity. It applies, for instance, to two deployed copies of the same AI system type that occupy the same profile-relative state at every time. Their deployment conditions must therefore be sufficiently compatible for the copies to have identical quantified profiles and trustworthiness levels throughout the observed interval. That said, equality of the history functors does not require the represented tokens to undergo the same lifecycle interventions.

\subsection{Weak and Strong AI Identity Criteria}
\label{subsubsection:categorical_recovery_identity}
The identity criteria in
Definition~\ref{def:ferrario_identity_criteria}
admit \textbf{weak} and \textbf{strong} categorical readings. 

\begin{definition}[Weak and strong state identity] \label{def:weak_strong_state_identity} 
Fix the type datum \(\diam\). For states \(x,y\in S_\diam\), the \emph{weak} \(\diam\)-relative identity relation is defined by 
\[ x\equiv_\tau y \quad\Longleftrightarrow\quad \tau_\diam(x)=\tau_\diam(y). \] 
The \emph{strong} \(\diam\)-relative identity relation is categorical isomorphism in \(\Sys_\diam\): 
\[ 
x\cong_{\Sys_\diam} y. 
\] 
\end{definition}

On the \textbf{weak} reading, two states are identical relative to \(\diam\) whenever they belong to the same trustworthiness-level fibre. On the \textbf{strong} reading, identity requires categorical isomorphism in \(\Sys_\diam\), that is, admissible trustworthiness-level-preserving reachability in both directions. 
However, for synchronic comparisons between realized histories, this strong condition must additionally be witnessed by time-synchronous representatives. Thus, ambient isomorphism in \(\Sys_\diam\) is necessary but does not by itself establish strong synchronic identity in
\(\Traj_{\diam,[t_0,T]}\). While the weak reading directly recovers the criterion of \citet{ferrario2025trustworthiness}---see Definition~\ref{def:ferrario_identity_criteria}, the strong reading is a stricter category-theoretic refinement that follows from our formalization.

The relations in Definition~\ref{def:weak_strong_state_identity} are defined on the abstract state space \(S_\diam\). However, AI system identity  concerns states occupied by realized AI system histories. In other words, the relevant states are of the form \(\mathsf A_\diam(t)\) for some realized history \(\mathsf A_\diam\in\operatorname{Ob}\bigl(\Traj_{\diam,[t_0,T]}\bigr)\) and some \(t\in[t_0,T]\). These considerations motivate the following theorem.

\begin{theorem}[Weak and strong identity for realized AI system histories]
\label{thm:ferrario_identity_categorical}
Let
\(\mathsf A_\diam,\mathsf B_\diam\)
be objects of
\(\Traj_{\diam,[t_0,T]}\), and write
\[
a(t)=\mathsf A_\diam(t),
\qquad
b(t)=\mathsf B_\diam(t).
\]
Diachronic identity compares states \(a(t)\) and \(a(t')\) along the same realized history, while synchronic identity compares states \(a(t)\) and \(b(t)\) at the same time \(t\). Then the following statements hold.

\begin{enumerate}[label=(\roman*),nosep]

\item For every \(t\leq t'\), the history morphism
\(
\mathsf A_{t,t'}\in\Hom_{\Sys_\diam}(a(t),a(t'))
\)
in \(\Sys_\diam\) implies
\[
a(t)\equiv_\tau a(t').
\]
Hence every realized AI system history satisfies the
\textbf{\emph{weak diachronic identity criterion}}.

\item For \(t\leq t'\), the realized states \(a(t)\) and \(a(t')\) satisfy
\textbf{\emph{strong diachronic identity}},
\[
a(t)\cong_{\Sys_\diam} a(t'),
\]
if and only if there exists a morphism in \(\Sys_\diam\) from \(a(t')\) to \(a(t)\). Equivalently, the realized history \(\mathsf A_\diam\) satisfies strong diachronic identity throughout
\([t_0,T]\) if and only if it factors through the maximal subgroupoid \(
\operatorname{Core}(\Sys_\diam)
\hookrightarrow
\Sys_\diam.
\)

\item If, for some \(t\in[t_0,T]\), there exists a time-synchronous comparison morphism
\(
\eta_t\in\Hom_{\Sys_\diam}(a(t),b(t)),
\)
then
\[
a(t)\equiv_\tau b(t).
\]
Consequently, the existence of a morphism \(\eta:\mathsf A_\diam\Rightarrow\mathsf B_\diam\) in \(\Traj_{\diam,[t_0,T]}\) implies
\textbf{\emph{weak synchronic identity}} at every \(t\in[t_0,T]\).

\item At time \(t\), the realized states \(a(t)\) and \(b(t)\) satisfy
\textbf{\emph{strong synchronic identity}} if and only if 
\(a(t)\cong_{\Sys_\diam} b(t)\) is witnessed time-synchronously; equivalently, if and only if there exist morphisms
\[
\eta_t\in\Hom_{\Sys_\diam}(a(t),b(t)),
\qquad
\rho_t\in\Hom_{\Sys_\diam}(b(t),a(t)),
\]
each admitting a time-synchronous representative at \(t\). In particular,
\[
\mathsf A_\diam\cong_{\Traj_{\diam,[t_0,T]}}\mathsf B_\diam
\]
if and only if strong synchronic identity holds at every
\(t\in[t_0,T]\).

\end{enumerate}
\end{theorem}

\begin{proof}
Every morphism in \(\Sys_\diam\) admits at least one representative in
\(\mathsf{Path}^{\tau}_\diam\), and all states occurring along such a
representative have one common trustworthiness level. 

For~$(i)$, the history morphism
\(
\mathsf A_{t,t'}
\)
therefore implies
\(
\tau_\diam(a(t))
=
\tau_\diam(a(t')),
\)
and hence
\(
a(t)\equiv_\tau a(t').
\)

For~\((ii)\), the history morphism \(\mathsf A_{t,t'}\)
already provides trustworthiness-preserving reachability in the forward direction. Hence \(a(t)\cong_{\Sys_\diam}a(t')\) holds exactly when there is also a reverse morphism from \(a(t')\) to \(a(t)\) in \(\Sys_\diam\). Since \(\Sys_\diam\) is thin, the existence of morphisms in both directions forces their composites to be the corresponding identity morphisms. Thus the forward history morphism is an isomorphism in \(\Sys_\diam\). Requiring this for every \(t\leq t'\) is precisely the condition that \(\mathsf A_\diam\)
factors through the inclusion
\(
\operatorname{Core}(\Sys_\diam)
\hookrightarrow
\Sys_\diam.
\)

For~$(iii)$, a time-synchronous comparison morphism
\(
\eta_t\in\Hom_{\Sys_\diam}(a(t),b(t))
\)
is represented by a trustworthiness-level-preserving path. Consequently,
\(
\tau_\diam(a(t))
=
\tau_\diam(b(t)),
\)
and therefore
\(
a(t)\equiv_\tau b(t).
\)
If
\(\eta:\mathsf A_\diam\Rightarrow\mathsf B_\diam\)
is a morphism in
\(\Traj_{\diam,[t_0,T]}\),
the same argument applies to every component \(\eta_t\), yielding weak synchronic identity at every
\(t\in[t_0,T]\).

For~$(iv)$, suppose first that there exist time-synchronous morphisms
\[
\eta_t\in\Hom_{\Sys_\diam}(a(t),b(t)),
\qquad
\rho_t\in\Hom_{\Sys_\diam}(b(t),a(t)).
\]
Their composites are endomorphisms of \(a(t)\) and \(b(t)\). Since
\(\Sys_\diam\) is thin,
\(
\rho_t\circ\eta_t=\id_{a(t)}\), \(\eta_t\circ\rho_t=\id_{b(t)}\).
Thus,
\(
a(t)\cong_{\Sys_\diam} b(t)
\)
through time-synchronous comparison morphisms. Now suppose that strong synchronic identity holds at every
\(t\in[t_0,T]\). Since \(\Sys_\diam\) is thin, the comparison morphisms in each direction are unique whenever they exist. Then the two component families automatically define natural transformations
\[
\eta:\mathsf A_\diam\Rightarrow\mathsf B_\diam,
\qquad
\rho:\mathsf B_\diam\Rightarrow\mathsf A_\diam
\]
in
\(\Traj_{\diam,[t_0,T]}\).
Their composites are endomorphisms of
\(\mathsf A_\diam\) and \(\mathsf B_\diam\), respectively. Since
\(\Traj_{\diam,[t_0,T]}\) is thin, these composites are the corresponding identity natural transformations. Hence
\[
\mathsf A_\diam\cong\mathsf B_\diam
\quad\text{in}\quad
\Traj_{\diam,[t_0,T]}.
\]
The converse follows immediately because an isomorphism in
\(\Traj_{\diam,[t_0,T]}\) has time-synchronous components in both directions at every time.
\end{proof}

Theorem~\ref{thm:ferrario_identity_categorical} clarifies how the weak and strong readings behave once they are applied to realized AI system histories. First, weak diachronic identity is automatic for every realized AI system history, even though the quantified profile, implementation, deployment conditions, or lifecycle provenance may vary along the history. Second, strong diachronic identity is stricter. It amounts to mutual trustworthiness-preserving \emph{reachability} of the endpoint states. However, this does not mean that the realized lifecycle is temporally reversible: the reverse morphism from \(a(t')\) to \(a(t)\) is a morphism in \(\Sys_\diam\), not necessarily a time-ordered reversal of the sequence of interventions. Third, time-synchronous comparison between realized histories implies weak synchronic identity. Thus, morphisms in
\(\Traj_{\diam,[t_0,T]}\), i.e., time-synchronous natural transformations, provide a categorical sufficient condition for weak synchronic AI system identity throughout the observed interval. Finally, isomorphisms in
\(\Traj_{\diam,[t_0,T]}\) of realized AI system history functors over $[t_0,T]$ capture strong synchronic identity of realized states over time.

\subsection{Examples of AI Identity}
\label{subsec:examples}
\paragraph{Example 1: Weak and strong diachronic identity under lifecycle change.}
Consider a credit-scoring AI system deployed at time \(t\) and subsequently recalibrated, retrained on refreshed data, and migrated to a new serving infrastructure at time \(t'>t\). These interventions may change its quantified profile, so that
\(
Q_P^{\mathsf A}(t)\neq Q_P^{\mathsf A}(t'),
\)
while leaving its trustworthiness level unchanged. Then, the two states are weakly diachronically identical, \(a(t)\equiv_\tau a(t'),\) even though they are not equal. However, how much profile variation is compatible with weak identity depends on the granularity of \(L_P\): a finer level function distinguishes smaller changes, whereas a coarser function permits greater variation within one level fibre. Strong diachronic identity holds when a reverse admissible trustworthiness-level-preserving path from \(a(t')\) to \(a(t)\) also exists. Such a reverse morphism witnesses mutual trustworthiness-preserving reachability of the endpoint states. It need not undo the concrete interventions that produced the later state, nor need it represent a backward-in-time lifecycle trajectory. If an update instead produces
\[
\tau_\diam(a(t'))\neq\tau_\diam(a(t)),
\]
the identity-preserving history segment terminates. The later state cannot belong to the same AI system history functor, although it may become the initial state of a new history segment after reassessment, remediation, or reclassification. The two segments are neither equal as functors nor isomorphic through time-synchronous natural transformations.

\paragraph{Example 2: Weak synchronic identity under profile-relative rescaling.}
Let \(\mathsf A_\diam,\mathsf B_\diam:
\Time_{[t_0,T]}\to\Sys_\diam\) represent two copies of the same AI system type deployed in different but type-compatible settings, such as the same medical AI system used in two comparable clinics within the same geographic region. At time \(t\), write \(a(t)=\mathsf A_\diam(t)=(Q_P^{\mathsf A}(t),L_P(Q_P^{\mathsf A}(t)))\) and \(b(t)=\mathsf B_\diam(t)=(Q_P^{\mathsf B}(t),L_P(Q_P^{\mathsf B}(t)))\). Suppose that the two profiles differ only in the \(i\)-th profile dimension, with \(Q_{P,i}^{\mathsf B}(t)=\lambda Q_{P,i}^{\mathsf A}(t)\) for some \(\lambda\in(0,1)\), while \(Q_{P,j}^{\mathsf B}(t)=Q_{P,j}^{\mathsf A}(t)\) for every \(j\neq i\). The rescaling may represent reduced evidential support or deployment-specific attenuation of the \(i\)-th trustworthiness dimension.

Assume that \(L_P\) is insensitive to changes in the \(i\)-th coordinate along the rescaling range under consideration. It follows that
\(
L_P\!\left(Q_P^{\mathsf A}(t)\right)
=
L_P\!\left(Q_P^{\mathsf B}(t)\right)
\). Choose
\(
1=\lambda_0>\lambda_1>\cdots>\lambda_m=\lambda,
\)
and define profiles \(Q_P^{(k)}(t)\in I_P\) by
\[
Q_{P,j}^{(k)}(t)
=
\begin{cases}
\lambda_k Q_{P,i}^{\mathsf A}(t), & j=i,\\[2mm]
Q_{P,j}^{\mathsf A}(t), & j\neq i.
\end{cases}
\]
Let
\[
x_k(t)
=
\left(
Q_P^{(k)}(t),
L_P(Q_P^{(k)}(t))
\right),\quad L_P(Q_P^{(k)}(t))=L_P(Q_P^{\mathsf A}(t)),\text{ for all } k=0,\dots,m.
\]

Suppose, in addition, that every \(x_k(t)\) belongs to
\(S^{\mathrm{real}}_\diam(t)\)
and that each consecutive pair
\(x_{k-1}(t),x_k(t)\)
is connected by an admissible primitive transformation. Then
the path 
\[
f:\bigl(
x_0(t)=a(t)\to
x_1(t)\to\cdots\to
x_m(t)=b(t)
\bigr)
\]
is a time-synchronous trustworthiness-level-preserving representative of the morphism
\(\eta_t\in\Hom_{\Sys_\diam}(a(t),b(t))\). It witnesses categorical comparability at time \(t\) and therefore establishes weak synchronic identity between the two copies at that time.

\paragraph{Example 3: Strong synchronic identity under an invertible profile transformation.}
Consider a vendor-hosted production
system \(\mathsf A_\diam\) and an operationally different local, white-labeled, or canary deployment \(\mathsf B_\diam\). Let
\(\mathsf A_\diam,\mathsf B_\diam:
\Time_{[t_0,T]}\to\Sys_\diam\)
be their realized AI system histories, and write
\(
a(t)
=
\left(
Q_P^{\mathsf A}(t),
L_P(Q_P^{\mathsf A}(t))
\right).
\)
Let
\(\varphi:I_P\to I_P\)
be a bijection with inverse
\(\varphi^{-1}\), and suppose that
\[
L_P(\varphi(Q))=L_P(Q)
\qquad
\text{for every }Q\in I_P.
\]
Define
\[
Q_P^{\mathsf B}(t)
=
\varphi(Q_P^{\mathsf A}(t)),
\qquad
b(t)
=
\left(
Q_P^{\mathsf B}(t),
L_P(Q_P^{\mathsf B}(t))
\right).
\]

Suppose that the forward and inverse profile transformations are witnessed at time \(t\) by admissible time-synchronous representatives defining morphisms
\(
\eta_t\in\Hom_{\Sys_\diam}(a(t),b(t))\) and 
\(\rho_t\in\Hom_{\Sys_\diam}(b(t),a(t)).\)
Since \(\Sys_\diam\) is thin, 
\(
\rho_t\circ\eta_t=\id_{a(t)}\), \(\eta_t\circ\rho_t=\id_{b(t)},
\)
and hence
\(
a(t)\cong_{\Sys_\diam} b(t).
\)
Thus, the two states satisfy strong synchronic identity at time \(t\). If such time-synchronous representatives exist in both directions for every \(t\in[t_0,T]\), then the corresponding component families assemble into a natural isomorphism \(\mathsf A_\diam\cong_{\Traj_{\diam,[t_0,T]}}\mathsf B_\diam,\) and the two realized histories satisfy strong synchronic identity throughout the observed interval.

The bijectivity of \(\varphi\) at the profile level does not by itself produce morphisms in \(\Sys_\diam\). Strong synchronic identity additionally requires that both directions be instantiated by admissible time-synchronous lifecycle paths. Governance-relevant examples include the identity transformation \(\varphi=\id_{I_P}\), for which \(\mathsf B_\diam=\mathsf A_\diam\) and \(\eta_t=\id_{a(t)}\). A nontrivial example is the coordinatewise power transformation
\(\varphi_p(Q_1,\ldots,Q_n)=((Q_1)^p,\ldots,(Q_n)^p)\), with \(p>0\), whose inverse is \(\varphi_{1/p}\). Such a transformation may represent an invertible nonlinear rescaling of the profile measurements: \(p>1\) gives greater prominence to high scores and compresses intermediate ones, whereas \(0<p<1\) expands intermediate and lower scores. More generally, one may use
\[
\varphi_{p}(Q_1,\ldots,Q_n)=((Q_1)^{p_1},\ldots,(Q_n)^{p_n}),
\]
where every \(p_i>0\), with inverse given by the exponents \(1/p_i\). Setting \(p_i=1\) on selected coordinates allows only particular profile dimensions to be rescaled. These transformations can model invertible changes in normalization, reporting conventions, or dimension-specific governance sensitivity admitted by the measurement and normalization rules fixed within \(P\). A change of measurement conventions not already covered by \(P\) would instead change the type datum and would not be represented within the same category \(\Sys_\diam\). Such transformations  support strong synchronic identity only when the level function is invariant under the transformation and the forward and inverse transformations admit time-synchronous lifecycle representatives.

\section{Discussion}
\label{section:discussion}
We developed a categorical account of AI system identity from a trustworthiness-based metaphysics of artifacts. The main contribution of the account is that it separates relations between AI system states that the original propositional criteria in \citep{ferrario2025trustworthiness} leave undifferentiated. While the biconditionals defining \(=_\diam\) in \citep{ferrario2025trustworthiness} identify AI systems, synchronically and diachronically, through equality of trustworthiness levels, the categorical construction recovers this weak criterion, but also adds directed and mutual reachability, temporally admissible histories, and time-synchronous comparison of realized AI system histories.

\subsection{A Hierarchy of Identity and Reachability Relations}
Table~\ref{tab:identity_relation_schema} summarizes the identity-relevant relations introduced by the formalization. The original relation \(=_\diam\) is a type-relative equivalence relation induced by equality of trustworthiness levels. The categorical weak relation \(\equiv_\tau\) recovers this level-theoretic criterion at the level of profile-relative abstract states. The preorder \(\preceq_\tau\) records directed trustworthiness-level-preserving reachability between abstract states: \(x\preceq_\tau y\) means that at least one admissible level-preserving path from \(x\) to \(y\) exists. Strong state-level identity is captured by isomorphism in \(\Sys_\diam\), that is, by mutual state reachability. Strong synchronic identity arises at the level of isomorphic realized history functors 
\[
\mathsf A_\diam\cong_{\Traj_{\diam,[t_0,T]}}\mathsf B_\diam.
\]
Thus, our formalization realizes a characteristically categorical idea: identity is not read off from internal descriptions alone, but from structure-preserving transformations \citep{maclane1965categorical,mac1971categories}, and strong synchronic identity of AI systems is expressed by transformations between functors that encode their evolution over time.

\begin{table}[ht]
\centering
\footnotesize
\renewcommand{\arraystretch}{1.25}
\begin{tabularx}{\textwidth}{
    >{\raggedright\arraybackslash}p{0.18\textwidth}
    >{\raggedright\arraybackslash}p{0.21\textwidth}
    >{\raggedright\arraybackslash}X
}
\toprule
\textbf{Relation} & \textbf{Level} & \textbf{Interpretation} \\
\midrule

\(=_\diam\) &
Propositional AI identity criterion &
The original type-relative criterion of \citet{ferrario2025trustworthiness}. Synchronically and diachronically, it is defined by equality of trustworthiness levels within the fixed type datum \(\diam\). \\

\(\equiv_\tau\) &
Weak categorical identity &
For \(x,y\in S_\diam\),
\[
x\equiv_\tau y
\quad\Longleftrightarrow\quad
\tau_\diam(x)=\tau_\diam(y).
\]
Applied to realized states, this yields weak diachronic identity \(a(t)\equiv_\tau a(t')\) and weak synchronic identity \(a(t)\equiv_\tau b(t)\). \\

\(\preceq_\tau\) &
Directed reachability &
For \(x,y\in S_\diam\),
\[
x\preceq_\tau y
\Longleftrightarrow
\text{there exists a trustworthiness-level-preserving path }x\rightsquigarrow y.
\]
This is an abstract reachability relation. It does not, by itself, imply temporal ordering, causal realization, or operational recoverability. \\

\(x\cong_{\Sys_\diam}y\) &
Strong state-level identity &
State isomorphism in \(\Sys_\diam\), equivalently mutual trustworthiness-preserving reachability: \(x\preceq_\tau y\) and \(y\preceq_\tau x\).
For realized diachronic comparison, \(a(t)\cong_{\Sys_\diam}a(t')\) expresses strong identity of endpoint states.\\

\(\eta:\mathsf A_\diam\Rightarrow\mathsf B_\diam\) &
Weak synchronic comparison of histories &
A morphism in \(\Traj_{\diam,[t_0,T]}\) consists of time-synchronous comparison morphisms \(\eta_t\in\Hom_{\Sys_\diam}(a(t),b(t))\) at every \(t\in[t_0,T]\). Its existence implies weak synchronic identity \(a(t)\equiv_\tau b(t)\)
throughout the interval. \\

\(\mathsf A_\diam\cong_{\Traj_{\diam,[t_0,T]}}\mathsf B_\diam\) &
Strong synchronic identity of realized histories &
Natural isomorphism of realized AI system history functors. It holds precisely when, at every \(t\in[t_0,T]\), the realized states \(a(t)\) and \(b(t)\) are mutually comparable through time-synchronous trustworthiness-preserving morphisms. \\

\bottomrule
\end{tabularx}
\caption{Hierarchy of identity, provenance, and reachability relations. Weak identity recovers equality of trustworthiness levels. Strong state-level identity adds mutual reachability in \(\Sys_\diam\). Strong synchronic identity of realized histories is captured by natural isomorphism in \(\Traj_{\diam,[t_0,T]}\).}
\label{tab:identity_relation_schema}
\end{table}

This hierarchy also explains why the construction does not stop at the path categories. The free path category \(\mathsf{Path}_\diam\) is too large for identity because it contains arbitrary admissible lifecycle paths, including paths that leave a trustworthiness-level fibre and later return to it. Restricting to \(\mathsf{Path}^{\tau}_\diam\) solves this problem by retaining only paths whose intermediate states preserve one trustworthiness level. However, \(\mathsf{Path}^{\tau}_\diam\) still remembers provenance, namely the particular recorded sequence of lifecycle transformations by which a state is reached. Distinct level-preserving paths between the same source and target remain distinct, and mutual reachability does not yet amount to categorical isomorphism: the composites of a path \(x\rightsquigarrow y\) and a path \(y\rightsquigarrow x\) are generally non-empty loops, not identity morphisms. Stopping at \(\mathsf{Path}^{\tau}_\diam\) would therefore yield a provenance-sensitive theory of level-preserving lifecycle paths, not a categorical theory of identity as isomorphism.

The quotient \(\Sys_\diam\) performs the required identity-relevant abstraction as it identifies parallel trustworthiness-level-preserving paths and retains only the fact that one state is reachable from another within the same trustworthiness-level fibre. We displayed this graphically in the upper parts of Figures~\ref{fig:AI_history_equiv_class} and~\ref{fig:trustworthiness_level_drop}. However, in \(\Sys_\diam\), an isomorphism does not assert that one concrete lifecycle process reverses another. A rollback, retraining, or recalibration may reconstruct a previous profile-relative state without undoing every computational, organizational, informational, or external consequence of the forward transformation. That is, isomorphism in \(\Sys_\diam\) means only that \emph{both directed, trustworthiness-preserving reachability relations exist}. Since \(\Sys_\diam\) is thin, the two composites are equal to the corresponding identity morphisms in the quotient category. These identity morphisms are represented by empty paths, but the lifecycle loops witnessing mutual reachability need not be empty and need not undo one another.

Similarly, natural isomorphism in \(\Traj_{\diam,[t_0,T]}\) does not require realized AI system histories to coincide. It requires, for every time \(t\in[t_0,T]\), time-synchronous comparison morphisms
whose composites are equal to the identity morphisms in \(\Sys_\diam\). Naturality is automatic once these components exist, because \(\Sys_\diam\) is thin. Thus, natural isomorphism of two realized AI system histories expresses \emph{mutual same-time comparability of two realized AI system histories throughout the interval}, not equality of their raw lifecycle provenance. These natural isomorphisms can be constructed explicitly: a sufficient condition to obtain such  natural isomorphisms is to exhibit, at every time \(t\), trustworthiness-level-preserving invertible transformations between the assessed states of the two realized histories in both directions, as in Example~3.

\subsection{Trustworthiness-Level Changes and Identity Interruption}
In addition, our categorical approach gives a precise sense in which trustworthiness-level changes are identity-interrupting events. If two states \(x^{-},x^{+}\in S_\diam\) satisfy
\[
\tau_\diam(x^{-})\neq \tau_\diam(x^{+}),
\]
then no morphism from \(x^{-}\) to \(x^{+}\) exists in \(\Sys_\diam\). Hence \(x^{-}\) and \(x^{+}\) cannot both lie in the image of the same AI system history functor. They may both be states of AI systems of the same type datum \(\diam\), but they are functorially incommensurable with respect to an uninterrupted identity-preserving history. A change of trustworthiness level therefore terminates one history segment and, if a post-change system is realized, begins another. This point has an important consequence for the design of the trustworthiness-level function \(L_P\). The function \(L_P\) should not be so sensitive to ordinary variation, measurement noise, or expected operational fluctuation in \(q_P\) that minor changes repeatedly push the system across level boundaries. If this happens, the system may appear to \emph{flicker in and out of existence} with respect to a single identity-preserving history functor. This would make identity unstable, evidence transfer difficult to justify, and governance continuity practically unauditable \citep{ferrario2025trustworthiness,ferrario2026methodology}.

For this reason, \(L_P\) must be designed and validated as a robust governance tool. Its level boundaries should be tested during design under plausible operational scenarios, including measurement uncertainty, deployment variation, distribution shift, monitoring noise, and expected model updates \citep{ferrario2026methodology}. Changes in quantified assessments $q_P$ \emph{within} a trustworthiness-level fibre should represent tolerable profile variation, while boundary crossings should mark genuine governance-relevant transitions, such as \emph{substantial modifications} in the EU AI Act---see Article 3(23) therein \citep{EU_AI_Act_2024}. In practice, this may require margin conditions around level boundaries, scenario testing, and robustness analysis. Metaphysically, an unstable \(L_P\) creates unstable identity histories. From a governance perspective, level functions should be auditable for their capacity to support persistence judgments under realistic lifecycle variation.

\subsection{The Use of Categories and Compositionality}
One might ask whether the same formalization could be developed using graph-theoretic or transition-system language. The answer is partly affirmative: our construction begins with precisely such data, namely, the free path category \(\mathsf{Path}_\diam\). At this first level, graph theory and category theory are therefore closely aligned. However, the categorical language becomes important in the subsequent steps. It makes path composition explicit, isolates the level-preserving subcategory \(\mathsf{Path}^{\tau}_\diam\), supports quotienting by parallel level-preserving paths, and yields the thin reachability category \(\Sys_\diam\). It then allows post-deployment histories to be represented as functors from a time category into \(\Sys_\diam\), and synchronic comparison between histories to be represented by natural transformations. Thus, category theory provides a unified language for the successive abstractions needed by the AI identity problem: from primitive transformations, to composed lifecycle paths, to trustworthiness-level-preserving reachability, to time-indexed histories, and finally to comparison between histories. This is also where the construction recovers the original propositional identity criteria while making a stronger criterion available. The stronger criterion is expressed categorically: state-level identity is captured by isomorphism in \(\Sys_\diam\), and history-level identity by natural isomorphism in \(\Traj_{\diam,[t_0,T]}\). This is where our construction follows a classical categorical strategy \citep{mac1971categories} explicitly: instead of comparing objects only by their internal descriptions, it compares them through structure-preserving transformations, and then compares whole histories through transformations between \emph{functors}. This is aligned with Freyd's well-known characterization of category theory:\footnote{Freyd's characterization of categories is also recalled by Mac Lane in his discussion of categorical algebra \citep{maclane1965categorical}.}

\begin{quote}
[...] and category theory is likewise better described as the theory of functors. [...]
It is not too misleading, at least historically, to say that categories are what one must define in order to define functors, and that functors are what one must define in order to define natural transformations \citep[p.~1]{freyd1964abelian}
\end{quote}

Finally, our categorical construction deliberately separates
\emph{realization} of specific AI system tokens from \emph{reachability} of abstract states populated by these tokens over time. In doing so, we have not  introduced a monoidal structure on the category of AI system histories. Monoidal categorical structures are tools for modeling \emph{compositionality} in mathematics and computer science, such as in the case of concurrent resources, and interacting processes \citep{fong2018seven,spivak2014category}. They would be natural if the aim were to model composition of AI system tokens, or the dynamics of parallel AI lifecycle resources.  The present paper focuses on abstract states, using only concatenation of lifecycle paths between abstract states representing populations of AI tokens, composition in the thin reachability category, functorial composition over temporal intervals, and  composition of natural transformations. Compositional operations on AI system states are a distinct problem considered in related work \citep{ferrario2026convex}.

\subsection{Metaphysical Neutrality About Persistence}
The present formalism does not require the objects \(\mathsf A_\diam(t)\) to be temporal parts, stages, or complete AI system tokens. Any state \(\mathsf A_\diam(t)\)  may be occupied by several concrete systems with different models, hardware, software stacks, or deployment histories. Likewise, the functor \(\mathsf A_\diam\) represents an AI system history rather than defining the system as a mereological sum of instantaneous entities. A perdurantist interpretation remains possible: one may understand a history as organizing temporally ordered stages or temporal parts. But an endurantist interpretation is also possible: one may understand the same functor as representing how one persisting system bears different profile-relative properties over time. A stage-theoretic reading would require the additional claim that each state is itself a complete AI system individual. However, no such claim follows from the construction.

The proposed framework is therefore neutral among the principal metaphysical theories of persistence. Its narrower commitment is that AI system identity under change requires temporally indexed profile-relative states, admissible lifecycle transformations, realized histories, and explicitly specified trustworthiness invariants. The framework tells us what must remain stable, reachable, comparable, or mutually reachable relative to \(\diam\). That said, it does not settle whether the persisting AI system is best understood as an enduring continuant, a perduring four-dimensional entity, or a sequence of stages.

\subsection{Recognizing Identity Relations in Governance Practice}
A further question concerns how these identity relations can be recognized in operational settings. This is an epistemological and governance-oriented question, which is not at the core of the present work. Nonetheless, we will briefly comment on it. The application of the present formalization depends on the forms of documentation, monitoring, and lifecycle evidence available for determining whether the relevant identity conditions hold in practice. In fact, as identity is defined relative to a type datum \(\diam=(F,P,L_P)\), its application depends on artifacts that are increasingly required by AI governance regimes. In particular, the EU AI Act requires, for high-risk AI systems, documented risk-management processes, technical documentation kept up to date, automatic record-keeping and logging, documented quality-management procedures, and post-market monitoring systems that collect and analyse data on performance and compliance throughout the system's lifetime \citep{EU_AI_Act_2024}. Such materials can document the system's techno-function and intended use, the specification of the trustworthiness profile \(P\), the measurement and aggregation procedures producing \(q_P\), the trustworthiness-level function \(L_P\), state-monitoring records, and lifecycle transformations such as retraining, recalibration, threshold adjustment, rollback, or deployment-environment change. They can therefore provide partial evidence for weak identity, directed reachability, mutual reachability, and identity interruption through trustworthiness-level changes. However, the detailed methodology for recognizing these relations ``in the wild'' requires an epistemology of AI system identity for governance and MLOps practice, and lies beyond the scope of the present paper. Relatedly, this material supporting categorical identity provides evidence for comparisons between AI system tokens, but not a sufficient epistemic warrant for all responsible AI claims.

\section{Conclusion}
\label{section:conclusions}
AI systems change after deployment through retraining, recalibration, reconfiguration, monitoring interventions, and changes in their operational environments. These transformations make identity a central problem for responsible AI: without explicit criteria of sameness, it remains unclear when evidence transfers across AI tokens, when explanations and fairness claims remain applicable, or when an update creates a governance-relevant discontinuity. To address these questions formally, we develop a categorical account of AI system identity grounded in techno-function and trustworthiness. The formalization distinguishes weak from strong identity criteria, recovering and extending the approach of \citet{ferrario2025trustworthiness}. Its central claim is that AI system identity is function-plus-trustworthiness identity. Category theory renders this claim temporal and structural without reducing an AI system to a model, a material implementation, or a disconnected sequence of snapshots. The account also remains neutral between endurantist and perdurantist theories of persistence.

The framework provides a basis for analyzing identity-preserving updates, substantial modifications, and the transfer of evidence and accountability across versions and deployments. Its practical application requires explicit trustworthiness profiles, level functions, temporally coherent lifecycle records, and documented assumptions about admissibility and reversibility. Where these elements are absent, the identity of a changing AI system remains insufficiently specified for reliable AI governance and for well-grounded epistemological and ethical assessments of human--AI interaction.

\section*{Acknowledgments}
We acknowledge partial support by the Swiss National Science Foundation (SNSF), grant no. 229061. OpenAI's ChatGPT version GPT-5.6 Thinking was used for language editing, including grammar and typo detection, improvements to English expression, and assistance in drafting and refining the TikZ code for the figures. The author reviewed and validated all outputs and remains solely responsible for the manuscript's conceptual,
mathematical, and content.

\bibliographystyle{plainnat} 
\bibliography{sample-base}

\end{document}